\documentclass[11pt,reqno]{amsart}
\usepackage[utf8]{inputenc}
\usepackage[T1]{fontenc}
\usepackage{lmodern}
\usepackage[english]{babel}
\usepackage{amsmath,a4wide}
\usepackage{xfrac}
\usepackage{esint}
\usepackage{stmaryrd,mathrsfs,bm,amsthm,mathtools,yfonts,amssymb,color,braket,booktabs,graphicx,graphics,amsfonts}
\usepackage{latexsym,microtype,indentfirst,hyperref}
\usepackage{xcolor}
\usepackage{courier}
\usepackage{constants}
\usepackage{tikz}

\newtheorem{theorem}{Theorem}[section]
\newtheorem{corollary}[theorem]{Corollary}

\newtheorem{lemma}[theorem]{Lemma}
\newtheorem{proposition}[theorem]{Proposition}

\theoremstyle{definition}
\newtheorem{definition}[theorem]{Definition}
\newtheorem{remark}[theorem]{Remark}

\numberwithin{equation}{section}
\numberwithin{subsection}{section}

\newcommand{\R}{\mathbb{R}} %reali
\newcommand{\spt}{\mathrm{spt}} %support (usually of a measure)
\newcommand{\bG}{\mathbf{G}} %Grassmannian

\newcommand{\rC}{\mathrm{C}} %cylinder
\newcommand{\bc}{\mathbf{c}}

\newcommand{\ssubset}{\subset\joinrel\subset}

\newconstantfamily{eps}{
symbol=\varepsilon}
\newconstantfamily{con}{
symbol=c}
\newconstantfamily{Lam}{
symbol=\Lambda}

\title{End-time regularity theorem for Brakke flows}
\date{\today}

\author[S. Stuvard]{Salvatore Stuvard}
\address{Dipartimento di Matematica, Universit\`{a} degli Studi di Milano, Via Saldini 50, I-20133 Milano (MI), Italy}
\email{salvatore.stuvard@unimi.it}

\author[Y. Tonegawa]{Yoshihiro Tonegawa}
\address{Department of Mathematics, Tokyo Institute of Technology, 2-12-1 Ookayama, Meguro-ku, Tokyo 152-8551, Japan}
\email{tonegawa@math.titech.ac.jp}

\begin{document}

\begin{abstract}
For a general $k$-dimensional Brakke flow in
$\mathbb R^n$ locally close to a $k$-dimensional plane in the
sense of measure, it is proved that the flow is represented locally as a smooth graph over the plane with estimates on all the derivatives up to the end-time. 
Moreover, at any point in space-time
where the Gaussian density is close to $1$, the flow can be
extended smoothly as a mean curvature
flow up to that time in a neighborhood: this extends White's local regularity 
theorem to general Brakke flows. The 
regularity result is in fact obtained for more
general Brakke-like flows, driven by the mean curvature plus an additional forcing term in a dimensionally sharp 
integrability class or in a H\"{o}lder
class.
\end{abstract}

\maketitle

\section{Introduction} 
A family of $k$-dimensional surfaces $M_t\subset \mathbb R^n$ parameterized by time $t$
is a mean curvature flow
(abbreviated as MCF) if the normal velocity is equal to the mean curvature
vector of $M_t$. Given a smooth $k$-dimensional submanifold $M_0$, there exists a unique smooth MCF with initial datum $M_0$ until singularities such as vanishing or neck-pinching occur. To extend the flow beyond the time of singularity,
numerous notions of generalized solution to MCF have been proposed 
since the 1970's: we mention, among others, the viscosity solutions produced by the level set method \cite{CGG,ES}, BV solutions \cite{Luckhaus}, and varifold solutions \cite{Brakke,Ton1}.

In the present paper, we focus on the varifold solutions known as Brakke flows, proposed and studied in Brakke's pioneering work 
\cite{Brakke}. One of the main results of \cite{Brakke} is
the partial regularity theorem of Brakke flows \cite[6.12]{Brakke}, which states that
any unit density Brakke flow is a smooth MCF for a.e.~time almost everywhere. Since a time-independent Brakke flow is a stationary varifold, and since in that case the unit density hypothesis means that the multiplicity function is equal
to $1$, the result may be seen as the natural parabolic
counterpart of the well-known result established by Allard in \cite{Allard} in the context of stationary varifolds.
For Brakke's partial regularity theorem,
as in many similar problems, the key ingredient is the proof of a ``flatness
implies regularity'' type result, that is, an $\varepsilon$-regularity theorem.
This is referred to as Brakke's local regularity theorem \cite[6.11]{Brakke} in this context.
It states, roughly speaking, that if $\{M_t\}_{t\in(-\Lambda,\Lambda)}$ 
is a Brakke flow in a cylinder $$\rC_2:=\rC(\mathbb R^k\times\{0\},2):=\{(x,y)\in\mathbb R^k\times\mathbb R^{n-k}:|x|<2\}$$ which is close to $$B_2^k:=\{(x,0)\in\mathbb R^k\times\mathbb R^{n-k}:
|x|<2\}$$
in the sense of 
measure over $t\in(-\Lambda,\Lambda)$, then, in the smaller cylinder $\rC_1$, $M_t$ coincides with a smooth graph over $B_{1}^k$ evolving by MCF for $t\in(-\Lambda/2,\Lambda/2)$, with estimates on all the derivatives of such graph
in terms of the overall height of $M_t$. The constant $\Lambda$ depends on how close $M_t$ is to $B_2^k$ in measure. While the original proof of Brakke's local regularity theorem contained various gaps and errors, 
a rigorous proof was provided in \cite{Kasai-Tone,Ton-2} with a different approach than Brakke's, and for more general flows, allowing for an additive perturbation in the form of a forcing term in the right-hand side of the underlying PDE.

Though this local regularity theorem is useful to prove the partial regularity of 
Brakke flows, there is a drawback in that it does not
provide the regularity of the flow up until the ``end-time''. Since the problem is parabolic in nature, one would expect the validity of interior estimates away from the ``parabolic boundary'' of $B_2^k \times (-\Lambda,\Lambda)$, and thus that the graphical representation over $B_1^k$ together with the corresponding estimates on the derivatives hold for $t \in (-\Lambda/2,\Lambda)$ instead of $(-\Lambda/2,\Lambda/2)$.

The present paper addresses precisely this problem, and 
proves that such estimates are possible for Brakke flows, even when the aforementioned forcing term is present. There are many more-or-less equivalent ways of stating the 
main regularity theorem proved here: an illustrative form is the following, where, for convenience, we discuss the simple case of Brakke flows with no forcing term
and we change 
the time interval from $(-\Lambda,\Lambda)$ to $[-2,0]$. For the sake of accuracy, the statement uses the varifold notation $V_t$ (see \cite{Allard,Kasai-Tone}), but the reader may think of
the support of the weight measure ${\rm spt}\|V_t\|$ as $M_t$. 

\begin{theorem}\label{main1}
Corresponding to $E_0 \in \left(0,\infty\right)$, there exists $\varepsilon_0=
\varepsilon_0(n,k,E_0)\in(0,1)$ with the following property. Suppose
$\{V_t\}_{t\in(-2,0]}$ is a $k$-dimensional unit density Brakke flow in the cylinder $\rC_3 = \rC(\R^k \times \{0\},3)
\subset\mathbb R^n$ satisfying: 
\newline
\indent
(1) $\sup_{t\in(-2,0]}\|V_t\|(\rC_3)\leq E_0$ ;
\newline
\indent (2) $\|V_{-4/5}\|(\rC_1)\leq \frac54\,\omega_k\,,$  ($\omega_k=
        \mbox{volume of }B_1^k$);  
        \newline
\indent        (3)
        $0\in {\rm spt}\|V_0\|$;
\newline
\indent        (4) $\cup_{t\in[-1,0]}\,{\rm spt}\|V_t\|\subset \{(x,y)\in\mathbb R^k\times\mathbb R^{n-k}:|y|\leq\varepsilon\}$ for some $\varepsilon\in(0, \varepsilon_0]$.
\newline
Then, for every $t \in \left[-1/4,0\right)$,  $\rC_{1/2}\cap{\rm spt}\|V_t\|$ is a $C^\infty$ graph over $B_{1/2}^k$ evolving by MCF, and
the space-time $C^\ell$-norm of the graph on $B_{1/2}^k\times[-1/4,0)$
is bounded by $c(\ell,n,k,E_0)\varepsilon$ for any $\ell\geq 1$. 
\end{theorem}

Any Brakke flow locally satisfies the assumption (1) for some $E_0>0$. The assumption (2)
excludes the case of two parallel $k$-dimensional
planes, which is not a univalent graph, while (3)
excludes the sudden vanishing of Brakke flow 
before the end-time $t=0$. Since the definition of Brakke flow
allows such irregularity, (3) (or some variant of
similar nature) is necessary. The last (4) assumes
that the height is kept small for $t\in[-1,0]$. The conclusion is that the 
Brakke flow is a smooth graph away from the parabolic 
boundary, and all derivatives can be controlled in 
terms of the height. Note that ${\rm spt}\|V_0\|$
may not be a smooth surface due to a possible (partial)
sudden vanishing at $t=0$, but we can smoothly
extend ${\rm spt}\|V_t\|$ as $t\rightarrow 0-$
in $\rC_{1/2}$ due to the estimates. 
As anticipated, the main result of the present paper 
is in fact more general. Precisely, the assumptions on the flow can be relaxed in various ways. First, the unit density assumption can be entirely dropped, and the theorem can be stated requiring $\{V_t\}_{t \in (-2,0]}$ to be a $k$-dimensional \emph{integral} Brakke flow in $\mathrm{C}_3$, instead. The reason is that assumption (2) prevents the presence of higher multiplicity points in a slightly smaller parabolic region, as one can see using Huisken's monotonicity formula and a compactness argument, so that any $k$-dimensional integral Brakke flow satisfying (1)-(4) for sufficiently small $\varepsilon_0$ is necessarily unit density in a smaller parabolic region. Second, assumption (4) on the smallness of the height can be phrased in a weaker measure-theoretic sense: for the result to be valid, it is in fact sufficient that the (space-time) $L^2$-distance of the flow from the plane $\R^k\times\{0\}$ in $\rC_1 \times [-1,0]$ (a quantity typically referred to as ($L^2$-)\emph{excess}) is sufficiently small. Furthermore, the regularity result proved here is in fact valid for the larger class of Brakke flows with forcing term; more precisely, in this case we obtain $C^{1,\zeta}$ 
($\zeta=1-k/p-2/q$) or $C^{2,\alpha}$ regularity estimates depending on whether the forcing is in the $L^{p,q}$-integrability class or in the
$\alpha$-H\"{o}lder class, respectively. There are several reasons, stemming both from theoretical considerations and from the applications, leading one to consider Brakke-like flows with additional forcing term. A major one is the study of Brakke flows on a Riemannian manifold $M$: once $M$ is (isometrically) embedded into some Euclidean space $\mathbb R^N$, the extrinsic curvatures of the immersion act as a forcing term in the corresponding definition of Brakke flow in $M$; see Section \ref{notation} for further details on this, and Theorems \ref{main2} and \ref{main3} for the precise statements of the main results.\\

We next discuss some related works. When the Brakke flow in Theorem \ref{main1} is a smooth
MCF or is obtained as a weak limit of smooth MCF, 
the result has been known as a part of White's local 
regularity theorem from \cite{Wh2}, and 
it has been used widely in the literature of MCF to 
analyze the nature of singularities. White's theorem applies, for instance, to Brakke flows obtained by the elliptic regularization method of 
Ilmanen \cite{Ilm1}, and, since the class of such MCF is weakly compact (see \cite[Section 7]{Wh2}), to their tangent flows. The present paper
shows that the same conclusions of White's theorem 
in various forms hold true even
without the proviso of approximability by
smooth MCF, and can be derived solely
from the definition of Brakke flow. As an 
illustration, using the main regularity theorem, we can prove
the following. 
\begin{theorem}\label{main1cor}
There exists $\Cl[eps]{ewhite}=\Cr{ewhite}(n,k)\in(0,1)$ 
with the following property. Let $\{V_t\}_{t\in (a,b]}$ be a
$k$-dimensional Brakke flow in a domain 
$U\subset \mathbb R^n$ (or an $n$-dimensional
Riemannian manifold). For any point $(x,t)\in U\times (a,b]$ with
the Gaussian density $\Theta(x,t)\in[1,1+\Cr{ewhite})$ (see Subsection \ref{app2}),
there exists $r>0$ such that $B_r(x)\cap {\rm spt}\|V_s\|$
is a smooth MCF in $B_r(x)$ 
for $s\in(t-r^2,t)$ and 
can be extended smoothly to $t$ in the limit as $s\rightarrow t-$. 
% There exist $\varepsilon_0=\varepsilon_0(n,k)\in(0,1)$
% and $c_0=c_0(n,k)\in(1,\infty)$ with the following
% property. Let $\{V_t\}_{t\in(-1,0)}$ be a $k$-dimensional
% Brakke flow in $B_2 \subset \R^n$ such that the Gaussian density 
% ratio (see ??? for the precise definition) satisfies
% \begin{equation}
%     \Theta((x,t),r)\leq 1+\varepsilon_0
% \end{equation}
% for all $(x,t)\in B_1\times(-1,0)$ and all $r \in \left(0, d(x,t)\right)$, where $d(x,t):=\min\{1-|x|,|t+1|^{\frac12}\}$ is the distance to the parabolic boundary of $B_1 \times (-1,0)$. Assume that the flow is non-trivial in $B_1$, so that
% $\int_{-1}^0\|V_t\|(B_1)\,dt>0$. 
% Define the space-time support of the Brakke flow as
% $\mathcal M:={\rm spt}(\|V_t\|\times dt)$ in
% $B_1\times(-1,0]$. Then any point $(x,t)\in \mathcal M$ 
% has a space-time neighborhood in $\mathbb R^n\times\mathbb R$
% in which $\mathcal M$ is an 
% embedded smooth MCF or it is smooth up to $t$ and
% becomes empty immediately after. 
% % Then there are 
% % at most countably many
% % $C^{\infty}$ connected
% % MCF $\{M^{(j)}_t\}_{t\in (-1,t_j]}$ with $t_j\in (-1,0]$ ($j\in I\subset\mathbb N$) such that 
% % \begin{equation}\label{meseq}
% % \|V_t\|=\sum_{j\in I}\mathcal H^k\lfloor_{M^{(j)}_t}\mbox{ 
% % for all }t\in(-1,0] \mbox{ in }B_1.
% % \end{equation}
% Moreover, the second
% fundamental form of the time slice $\mathcal M_t:=
% \{x:(x,t)\in\mathcal M\}$ at $x$
% is bounded by $c_0/d(x,t)$. 
\end{theorem}
\noindent

We remark that there are, in the literature, existence theorems of Brakke flows
for which one cannot tell {\it a priori} whether they arise as weak
limits of smooth MCF or not. The examples include the limits of solutions to the Allen-Cahn equation
\cite{Ilm_AC,Ton3,Takasao-Tonegawa} as well as the flows obtained by means of time-discrete approximate schemes \cite{Brakke,KimTone,ST19,ST22}.   
In the case of Brakke flows with no forcing term, 
Lahiri \cite{Lahiri} showed an analogous end-time $C^{1,\zeta}$
regularity result using some height growth estimates, a suitable 
constancy theorem for integral varifolds, and 
higher order derivative estimates. The proof is very different from that
of the present paper, and it appears difficult to generalize it
to flows with forcing term. More recently, Gasparetto \cite{gas}
showed the validity of a similar end-time $C^{1,\zeta}$ regularity result for Brakke flows with boundary, with a proof based on viscosity techniques. About six months after the present manuscript was made available as a preprint, De Philippis, Gasparetto, and Schulze provided in \cite{DGS} an alternative proof -- again based on viscosity techniques -- of the end-time regularity result in the interior for Brakke flows possibly with forcing term in $L^\infty$.\\

Next, we describe the idea of the proof. The proof of Theorem \ref{main1} is achieved by modifying suitable
portions of the proof of the local regularity theorem in \cite{Kasai-Tone}, so to extend the graphicality and the relevant estimates up to the end-time. Just 
as in many similar problems of this type, a fundamental step towards regularity is the proof of a Caccioppoli-type estimate stating that a certain ``Dirichlet-type energy'' can be controlled in terms of the $L^2$-height of the solution. In the context of Brakke flows, such Dirichlet type energy corresponds, roughly speaking, to the difference (excess) of surface measure of $\|V_t\|$ within the cylinder $\rC_1$ and the measure $\omega_k$ of the unit disk. Such difference is shown to be less than a constant times the $L^2$-height of the flow by means of an ODE argument, see \cite[Section 5]{Kasai-Tone}: indeed, 
one proves,
%it is not difficult to see, 
by appropriately testing Brakke's inequality, that the excess of measure -- as a function of time -- satisfies an ordinary differential inequality. The ODE argument implemented in \cite{Kasai-Tone}, though, requires some ``waiting time'' both near the beginning and the end of the time interval: this is the main reason for the lack of estimates up to the end-time in \cite{Kasai-Tone}. A key point of the present paper is the observation that such waiting time becomes
shorter when the height of the Brakke flow is smaller. 
The proof of the regularity then proceeds just like in Allard's regularity theorem: the Brakke flow is approximated by a (parabolic) Lipschitz function, and one initiates a blow-up argument. The approximating Lipschitz functions are rescaled by the height of the Brakke flow, but, thanks to the above mentioned key observation, in the process of passing to the limit as the height goes to $0$, also the waiting time becomes infinitesimal. One can then show that the rescaled Lipschitz functions converge strongly in $L^2$ 
to a solution of the heat equation as long as small neighborhoods of $t=-1$ and $t=0$ are removed. The contribution to the $L^2$-norms of the rescaled functions coming from the neighborhood
of $t=0$ can be made small, so that, in combination with the linear 
regularity theory of the heat equation, one obtains decay estimates for the linearized problem. By iterating, one concludes $C^{1,\zeta}$ regularity and graphical representation of $\|V_t\|$ on a parabolic region of space-time which touches the origin. In particular, 
any point on the boundary of this parabolic region is in the support of $\|V_t\|$, so that one can repeat the same argument regarding these
points as the origin. This implies that the domain of graphicality with estimates can be extended so that it covers the whole 
support of $\|V_t\|$ in $\rC_{\sfrac12}\times[-1/4,0)$, 
proving the $C^{1,\zeta}$ estimate up to the end-time. 
Once this is done, $C^{2,\alpha}$ regularity up to the end-time can be obtained by repeating -- with essentially no changes -- the proof in \cite{Ton-2}. 
Once the $C^{2,\alpha}$ end-time regularity is available, the classical 
parabolic regularity theory gives all the higher derivative 
estimates for the Brakke flow with no forcing term, while the regularity theory for inhomogeneous linear heat equation implies the result when the forcing term is present.\\ 

The paper is organized as follows. In Section \ref{notation} we set up the notation in use throughout the paper, and we provide the formal statements of the main results in their full generality (see Theorem \ref{main2} and Theorem \ref{main3}) as well as the proofs of Theorems \ref{main1} and \ref{main1cor} as a consequence of the general main results. Section \ref{s:en_est} contains the enhanced ODE argument which gives energy estimates with short waiting time at the end of the time interval. In Section \ref{s:Lip_approx} we produce a parabolic Lipschitz approximation of the flow with good estimates up to the end-time, by suitably modifying the corresponding construction in \cite[Section 7]{Kasai-Tone}. In Section \ref{buarg},
the main modification of the blow-up 
argument is described and the main $C^{1,\zeta}$ 
regularity on a parabolic domain touching
the origin (a subdomain of $\{(x,t): |x|^2<|t|\}$) is obtained. In Section \ref{PMR}, we complete the proof of Theorem \ref{main2} and Theorem \ref{main3}.

\medskip

\noindent\textbf{Acknowledgements.} The second author dedicates the present paper to his mother Suzuko Tonegawa. The work was carried out during the sabbatical year of the second author at the University of Milan. 

The research of S.S. was supported by the project PRIN 2022PJ9EFL "\textit{Geometric Measure Theory: Structure of Singular Measures, Regularity Theory and Applications in the Calculus of Variations}", funded by the European Union under NextGenerationEU and by the Italian Ministry of University and Research, and by the Gruppo Nazionale per l'Analisi Matematica, la Probabilit\`a e le loro Applicazioni of INdAM. Y.T. was partially supported by JSPS 18H03670, 19H00639.

\section{Assumptions and main results}\label{notation}
\subsection{Notation}
Since the proof follows \cite{Kasai-Tone}
very closely,
we mostly adopt the same notation (see \cite[Section 2]{Kasai-Tone}), except for a few symbols of norms. Throughout $1\leq k<n$ are fixed, and the dependence of constants on $n$ and $k$ is often not specified for simplicity. We set $\mathbb R^+:=\{x \in \R\,:\,x\geq 0\}$. 
For $r\in(0,\infty)$ and $a\in\mathbb R^n$ (or $a\in\mathbb R^k$) we set
\begin{equation*}
    B_r(a):=\{x\in\mathbb R^n : |x-a|<r\},\,\,
    B_r^k(a):=\{x\in\mathbb R^k : |x-a|<r\}\,,
\end{equation*}
and we often identify $\mathbb R^k$ with 
$\mathbb R^k\times\{0\}\subset\mathbb R^n$. 
When $a=0$, we may write $B_r$ and $B_r^k$. 
For $a\in\mathbb R^n$, $s\in\mathbb R$ and $r>0$ we define two types of parabolic cylinders
\begin{equation}\label{defpara}
\begin{split}
   & P_r(a,s):=\{(x,t)\in\mathbb R^n\times\mathbb R: |x-a|<r,\, |t-s|<r^2\}\,, \\
   & \tilde P_r(a,s):=\{(x,t)\in\mathbb R^n\times\mathbb R:|x-a|<r,\, s-r^2<t<s\}\,;
\end{split}
\end{equation}
the first one was used in \cite{Kasai-Tone}, whereas in the present paper we will prefer to work with the second one. We denote by $\mathcal L^n$ the Lebesgue measure 
on $\mathbb R^n$ and by $\mathcal H^k$ the $k$-dimensional Hausdorff measure on $\mathbb R^n$.
The restriction of a measure to a (measurable) set $A$ is 
expressed by $\lfloor_A$. For an open set
$U\subset\mathbb R^n$, $C_c(U)$ is the set of continuous and
compactly supported functions defined on $U$, and 
$C_c^k(U)$ is the set of $k$-times continuously differentiable
functions with compact support in $U$. The symbols
$\nabla f$ and $\nabla^2 f$ always denote the spatial gradient and Hessian of 
$f$, respectively, and $f_t=\partial_t f$ is the time derivative of $f$. For a function $f$
defined on a domain in space-time $D\subset\mathbb R^n\times\mathbb R$ and $\alpha\in(0,1)$, define
the following (semi-)norms to ease the notation in
\cite{Kasai-Tone,Ton-2}: 
\begin{equation*}
    \|f\|_0:=\|f\|_{L^\infty(D)}\,,
\end{equation*}
\begin{equation*}
    [f]_{\alpha}:=\sup \left\lbrace
    \frac{|f(y_1,s_1)-f(y_2,s_2)|}{\max\{|y_1-y_2|,|s_1-s_2|^{\frac12}\}^\alpha} \, \colon \, (y_1,s_1),(y_2,s_2)\in D\,, \;  (y_1,s_1)\neq(y_2,s_2) \right\rbrace\,,
\end{equation*}
\begin{equation*}
    [f]_{1+\alpha}:=[\nabla f]_{\alpha}+\sup\left\lbrace\frac{|f(y,s_1)-f(y,s_2)|}{|s_1-s_2|^{\frac{1+\alpha}{2}}} \, \colon \, (y,s_1),(y,s_2)\in  D\,,\; s_1 \neq s_2 \right\rbrace\,.
\end{equation*}

Let ${\bf G}(n,k)$ be the space of $k$-dimensional linear subspaces of $\mathbb R^n$  and let ${\bf A}(n,k)$
be the space of $k$-dimensional affine planes in $\mathbb R^n$. For $S\in{\bf G}(n,k)$, we identify $S$ with the corresponding orthogonal 
projection matrix of 
$\mathbb R^n$ onto $S$. Let $S^{\perp}\in{\bf G}(n,n-k)$ be the orthogonal complement of $S$.
For $A\in{\rm Hom}(\mathbb R^n;\mathbb R^n)$, we define
the operator norm 
\begin{equation*}
    \|A\|:=\sup\{|A(x)|:x\in\mathbb R^n, \, |x|=1\}\,, 
\end{equation*}
and we often use this as a metric on ${\bf G}(n,k)$.
For $T\in{\bf G}(n,k)$, $a \in T$, and 
$r\in(0,\infty)$ we define the cylinder
\begin{equation*}
    \rC(T,a,r):=\{x\in\mathbb R^n:|T(x-a)|<r\}\,.
\end{equation*}
A general $k$-varifold on  
$U\subset \mathbb R^n$ is a 
Radon measure defined on
$G_k(U):=U\times{\bf G}(n,k)$ (see \cite{Allard,Simon} for a more comprehensive introduction), and the set of 
all general $k$-varifolds in $U$ is denoted by ${\bf V}_k(U)$. 
For $V\in{\bf V}_k(U)$, let $\|V\|$ be the weight 
measure of $V$
(with no fear of confusion with the operator norm), that is the measure defined on $U$ by
\begin{equation*}
    \|V\|(\phi):=\int_{G_k(U)}\phi(x)\,dV(x,S) \qquad \mbox{for every $\phi \in C_c(U)$}\,.
\end{equation*}

For a proper map $f\in C^1(\mathbb R^n;\mathbb R^n)$, the symbol $f_{\sharp} V$ denotes the push-forward of the varifold $V$ through $f$. We say that
$V\in{\bf V}_k(U)$ is a rectifiable varifold if 
there are some $\mathcal H^k$-measurable and countably $k$-rectifiable set $M\subset\mathbb R^n$ as well as a non-negative function $\theta \in L^1_{{\rm loc}}(\mathcal{H}^k\lfloor_M)$ such that
\begin{equation*}
V(\phi)=\int_{M}\phi(x,{\rm Tan}_x M)\,\theta(x)\,d\mathcal H^k(x) \qquad \mbox{for all $\phi\in C_c(G_k(U))$}\,,
\end{equation*}
and in such case we write $V = {\bf var}(M,\theta)$. Here, ${\rm Tan}_x M$ is
the approximate tangent space to $M$ at $x$, which exists for $\mathcal H^k$-a.e. $x\in M$. When $\theta(x)$
is integer-valued for $\mathcal H^n$-a.e.\,$x\in M$, $V$ is said to be an integral varifold. 
The set of all integral varifolds is
denoted by ${\bf IV}_k(U)$. When $\theta=1$ additionally, we say that $V$ is of unit density. 
For $V\in {\bf V}_k(U)$, $\delta V$ denotes the first variation of $V$ and $\|\delta V\|$ denotes the total variation of $\delta V$. When $\delta V$ is bounded and absolutely
continuous with respect to $\|V\|$, the Radon-Nikodym derivative (times $-1$), $-\delta V/\|V\|$, is denoted
by $h(V,\cdot)$ and is called the generalized mean curvature vector
of $V$. A fundamental geometric property of integral varifolds, of great importance for the analysis of Brakke flows, is Brakke's perpendicularity theorem \cite[Chapter 5]{Brakke}: if $V\in{\bf IV}_k(U)$ and
$h(V,\cdot)$ exists, then $S(h(V,x))=0$ for $V$-a.e. $(x,S)\in G_k(U)$.

For a one-parameter family of varifolds $\{V_t\}_{t\in[a,b]}$, we often use $\|V_t\|\times dt$ to represent the natural product measure
on $U\times[a,b]$; the latter is also expressed
as $d\|V_t\|dt$ within integration. 

Fix $\phi\in C^\infty([0,\infty))$ such that
$0\leq \phi\leq 1$,
\begin{equation}
    \phi(x)\left\{\begin{array}{ll}
    =1 & \mbox{ for }0\leq x\leq (2/3)^{1/k}, \\
    >0 & \mbox{ for }0\leq x<(5/6)^{1/k}, \\
    =0 & \mbox{ for }x\geq (5/6)^{1/k}.
    \end{array}\right.
\end{equation}
For $R\in(0,\infty)$, $x\in\mathbb R^n$ and $T\in{\bf G}(n,k)$, define
\begin{equation}
    \phi_{T,R}(x):=\phi(R^{-1}|T(x)|),\,\,
    \phi_T(x):=\phi_{T,1}(x)=\phi(|T(x)|)
\end{equation}
and set
\begin{equation}
    {\bf c}:=\int_T \phi_T^2(x)\,d\mathcal H^k(x).
\end{equation}
The functions $\phi_{T,R}$ and $\phi_T$ will be used as smooth 
test functions to gauge the measure deviation of $\|V\|$
away from
$T$ with 
multiplicity one. Notice that ${\bf c}$ is independent of $T$.

\subsection{Definition of Brakke flow}
Since in this paper we are mostly interested in the end-time regularity,
we consider time intervals of the form 
$[-\Lambda,0]$ with $\Lambda>0$ in the following. 
\begin{definition}\label{brakke-def}
Suppose that $U\subset\mathbb R^n$ is a domain and 
$1\leq k<n$. 
A family of varifold $\{V_t\}_{t\in [-\Lambda,0]}\subset {\bf V}_k(U)$
is a ($k$-dimensional) Brakke flow if the following
conditions are satisfied.
\begin{enumerate}
    \item[(1)] For a.e.~$t\in [-\Lambda,0]$, $V_t\in {\bf IV}_k(U)$.
    \item[(2)] For all $\tilde U\ssubset U$, we have
    \begin{equation}\label{unimass}
        \sup_{t\in[-\Lambda, 0]}\|V_t\|(\tilde U)<\infty\,.
    \end{equation}
    \item[(3)] For a.e.~$t\in[-\Lambda, 0]$, $\delta
    V_t$ is locally bounded and absolutely 
    continuous with respect to $\|V_t\|$, and thus $h(V_t,\cdot)$ exists. Furthermore, For all
    $\tilde U\ssubset U$, 
    \begin{equation}\label{hbound}
        \int_{-\Lambda}^0\int_{\tilde U} |h(V_t,x)|^2\,
        d\|V_t\|dt<\infty\,.
    \end{equation}
    \item[(4)]
    For all $\varphi\in C^1(U\times[-\Lambda,0];\mathbb R^+)$ with $\varphi(\cdot, t)\in C^1_c(U)$ for all
    $t\in[-\Lambda,0]$, and for all $-\Lambda\leq t_1<t_2\leq 0$, we have
    \begin{equation}\label{defBra}
    \begin{split}
        &\int_U\varphi(x,t_2)\,d\|V_{t_2}\|(x)-\int_U
        \varphi(x,t_1)\,d\|V_{t_1}\|(x) \\ 
        &\leq\int_{t_1}^{t_2}dt\int_{U} \{(\nabla\varphi(x,t)-h(V_t,x)\varphi(x,t))\cdot h(V_t,x)+\varphi_t(x,t)\}\,d\|V_t\|(x)\,.
        \end{split}
    \end{equation}
\end{enumerate}
\end{definition}
The condition (4) is a weak formulation of 
MCF due to Brakke \cite{Brakke}. While Brakke's
original formulation of \eqref{defBra} is in the form of a
differential inequality, nothing is lost if one works in this 
integral formulation. In fact, the latter is advantageous, in that it may easily 
accommodate the setting with additional 
unbounded forcing term as described in the next subsection. 

One may naturally consider a MCF and the corresponding notion of Brakke flow in a general
$n$-dimensional Riemannian manifold $M$. By Nash's 
isometric embedding theorem, we may always consider $M$ to be
a submanifold in a domain $U\subset \mathbb R^N$ for some sufficiently large $N$. A Brakke flow in $M$
can then be defined by asking ${\rm spt}\|V_t\|
\subset M$ for all $t$, (1)-(3), and by replacing the
inequality \eqref{defBra} by 
\begin{equation} \label{defBraMan}
   \begin{split}
        &\int_U\varphi(x,t_2)\,d\|V_{t_2}\|(x)-\int_U
        \varphi(x,t_1)\,d\|V_{t_1}\|(x) \\ 
        &\leq\int_{t_1}^{t_2}dt\int_{G_k(U)} \{(\nabla\varphi(x,t)-h(V_t,x)\varphi(x,t))\cdot (h(V_t,x)-H_M(x,S))+\varphi_t(x,t)\}\,dV_t(x,S).
        \end{split}  
\end{equation}
Here,  $H_M(x,S)=\sum_{i=1}^k{\bf B}_x(v_i,v_i)\in ({\rm Tan}_x M)^\perp$, where ${\bf B}_x(\cdot,\cdot)$
is the second fundamental form of $M\subset 
\mathbb R^N$ at $x\in M$ and the set
$\{v_1,\cdots,v_k\}$ is an orthonormal basis of 
$S\in{\bf G}(n,k)$. See \cite[Section 7]{Ton-2} for a further explanation. The term $H_M$ is 
already perpendicular to $M$ and, for all 
analytical purposes, can be regarded as a 
locally bounded forcing term $u$ as described in the next
subsection. 

\subsection{Assumptions}
The following assumptions are the same as \cite{Kasai-Tone},
and we list them for the reader's convenience.

For an open
set $U\subset\mathbb R^n$, suppose that we have a family of 
$k$-varifolds $\{V_t\}_{t\in[-\Lambda,0]}\subset{\bf V}_k(U)$ and a family of $(\|V_t\|\times dt)$-measurable
vector fields $\{u(\cdot,t)\}_{t\in[-\Lambda,0]}$ defined on 
$U$ and satisfying the following.
\newline
\indent
(A1) For a.e.\,$t\in[-\Lambda,0]$, $V_t$ is a unit density 
$k$-varifold. 
\newline
\indent
(A2) There exists $E_1\in[1,\infty)$ such that
\begin{equation}\label{denrat}
    \|V_t\|(B_r(x))\leq \omega_k r^k E_1\,\,
    \mbox{ for all }B_r(x)\subset U\mbox{ and }t\in[-\Lambda,0]\,.
\end{equation}
\newline
\indent
(A3) The numbers $p\in[2,\infty)$ and $q\in(2,\infty)$
satisfy
\begin{equation}\label{pq}
    \zeta:=1-\frac{k}{p}-\frac{2}{q}>0\,,
\end{equation}
and $u$ satisfies
\begin{equation}
    \|u\|_{L^{p,q}(U \times \left[-\Lambda,0\right])}:=
    \left(\int_{-\Lambda}^0\left(\int_U |u(x,t)|^p\,d\|V_t\|(x)
    \right)^{\frac{q}{p}}\,dt\right)^{\frac{1}{q}} <\infty\,.
\end{equation}
\newline
\indent
(A4) For all $\varphi\in C^1(U\times [-\Lambda,0];\mathbb R^+)$
with $\varphi(\cdot,t)\in C^1_c(U)$ for all $t\in[-\Lambda,0]$,
and for all $-\Lambda\leq t_1<t_2\leq 0$,
we have
\begin{equation} \label{Bra-ineq}
\begin{split}
    &\int_U \varphi(x,t_2)\,d\|V_{t_2}\|(x)-\int_U \varphi(x,t_1)\,d\|V_{t_1}\|(x)
    \\&\leq \int_{t_1}^{t_2}dt\int_U \{
    (\nabla\varphi(x,t)-h(V_t,x)\varphi(x,t))\cdot (h(V_t,x)+
    u^{\perp}(x,t))+\varphi_t (x,t)\}\,
    d\|V_t\|(x)\,.
    \end{split}
\end{equation}
Implicitly in the 
formulation of (A4),
 it is assumed that the first variation $\delta V_t$ of $V_t$
 is locally bounded and it is absolutely continuous with respect to $\|V_t\|$ (so that $h(V_t,x)$ exists) for a.e. $t\in[-\Lambda,0]$, 
 and that $h(V_t,x)\in L^2_{\rm loc}
 (U\times[-\Lambda, 0])$. 
For a.e. $t\in[-\Lambda,0]$, $u^\perp(x,t)$ is the projection of $u$ onto the orthogonal complement of the approximate tangent space
to $V_t$ at $x$, which exists for $\|V_t\|$-a.e. $x$ due to the integrality of $V_t$.
 The inequality \eqref{Bra-ineq}
characterizes formally that the normal velocity of the flow is equal to the
mean curvature vector $h$ plus $u^\perp$. When $u \equiv 0$, \eqref{Bra-ineq} simply becomes \eqref{defBra}, and thus $\{V_t\}_{t \in \left[-\Lambda,0\right]}$ is a Brakke flow. More generally, \eqref{Bra-ineq} includes the case when $\{V_t\}_{t \in \left[-\Lambda,0\right]}$ is a Brakke flow in a Riemannian manifold $M$, which corresponds to $u(x,t) := -H_M(x,{\rm Tan}_x \|V_t\|)$: indeed, as already explained, in this case $u(x,t) \in ({\rm Tan}_x M)^\perp$, and thus in particular $u(x,t) \in ({\rm Tan}_x \|V_t\|)^\perp$ given that $\spt\|V_t\| \subset M$ for all $t$. One technical 
point to add is that (A1) may be replaced, for all purposes
of the present paper, by 
\newline\indent
(A1') for a.e.\,$t\in[-\Lambda,0]$, $V_t\in {\bf IV}_k(U)$.
\newline
The reason for this is that the assumptions of the main theorems essentially allow only unit density varifolds. We will nonetheless adopt (A1) as our working hypothesis, in order to be consistent with \cite{Kasai-Tone}.  As already mentioned, there are in the literature various results guaranteeing the existence of (generalized) MCF (possibly with forcing term $u$) satisfying (A1)-(A4).

\subsection{Main results}
The first theorem is the basic $\varepsilon$-regularity
theorem, and it corresponds to a parabolic version
of Allard's regularity theorem; the 
second theorem gives a $C^{2,\alpha}$ estimate.
They are the end-time regularity counterpart of
\cite{Kasai-Tone} and \cite{Ton-2}, respectively.

\begin{theorem}\label{main2}
Corresponding to $\nu\in(0,1)$, $E_1\in[1,\infty)$, $p$ and 
$q$ satisfying \eqref{pq}, there
exist $\Cl[eps]{e1}\in (0,1)$ and $\Cl[con]{c1}\in(1,\infty)$ 
depending only on $n,k,p,q,\nu,E_1$ with the following
property. For $R\in(0,\infty)$, $T\in{\bf G}(n,k)$, and $U=\rC(T,2R)$,
suppose $\{V_t\}_{t\in[-R^2,0]}$ and $\{u(\cdot,t)\}_{t\in[-R^2,0]}$
satisfy (A1)-(A4). Suppose furthermore that we have
\begin{equation}\label{udef1}
    \|V_{-4R^2/5}\|(\phi_{T,R}^2)\leq (2-\nu)\,{\bf c} \,R^k\,,
\end{equation}
\begin{equation}\label{udef2}
    (\rC(T,\nu R)\times\{0\})\cap {\rm spt}(\|V_t\|\times dt)\neq \emptyset\,,
\end{equation}
\begin{equation}\label{udef3}
    \mu:=\left(R^{-(k+4)}\int_{-R^2}^0\int_{\rC(T,2R)}
    |T^{\perp}(x)|^2\,d\|V_t\|dt\right)^{\frac12}<\Cr{e1}\,,
\end{equation}
\begin{equation}\label{udef4}
\|u\|_{p,q}:=R^\zeta\|u\|_{L^{p,q}(\rC(T,2R)\times[-R^2,0])}<\Cr{e1}\,.
\end{equation}
Let $\tilde D:=\left(B_{R/2} \cap T\right) \times[-R^2/4,0)$. Then there are $C^{1,\zeta}$ 
functions
$f:\tilde D\rightarrow T^\perp$ and $F:\tilde D\rightarrow \mathbb R^n$ such
that $T(F(y,t))=y$ and $T^\perp(F(y,t))=f(y,t)$ for all $(y,t)\in\tilde D$,
\begin{equation} \label{udef5-12}
    {\rm spt}\|V_t\|\cap \rC(T,R/2)={\rm image}\,F(\cdot,t)
    \mbox{ for all }t\in[-R^2/4,0),
\end{equation}
\begin{equation}\label{udef5}
R^{-1}\|f\|_0+\|\nabla f\|_0
+R^\zeta[f]_{1+\zeta}
\leq \Cr{c1}\max\{\mu,\|u\|_{p,q}\}.
\end{equation}
\end{theorem}
As discussed in the Introduction, \eqref{udef1} excludes the possibility that $V_t$ consists of multiple sheets in $\rC(T,R)$, and it can replace the assumption that $V_t$ be unit density. Notice that \eqref{udef1} is stated as a property valid at time $-4R^2/5$; nonetheless, the validity of \eqref{Bra-ineq} implies that in fact $\|V_t\|(\phi_{T,R}^2)$ is an almost-decreasing function of $t$, even when the forcing term $u$ is present. As a consequence, the mass estimate in \eqref{udef1} remains valid when $\|V_{-4R^2/5}\|$ is replaced by $\|V_t\|$ for $t > -4R^2/5$, modulo replacing $\nu$ with $\nu' \in (0,\nu)$, provided $\Cr{e1}$ is sufficiently small depending on $\nu'$. The assumption \eqref{udef2} prevents sudden vanishing of the flow prior to the end-time. Finally, \eqref{udef3} is a smallness requirement on the (space-time) $L^2$-height of the flow, namely of the space-time $L^2$-distance of the flow from  the given $k$-dimensional plane $T$. We notice explicitly that, as a consequence of \eqref{udef5-12}-\eqref{udef5}, one can naturally extend $f$ and $F$ to $t=0$ as $C^{1,\zeta}$ functions.
Nonetheless, $\rC(T,R/2)\cap {\rm spt}\|V_0\|\subset {\rm image}\,
F$, but equality may not hold in general.  

When $u$ is
$\alpha$-H\"{o}lder continuous, we have the $C^{2,\alpha}$-regularity estimate as follows.
\begin{theorem}\label{main3}
Corresponding to $\nu\in(0,1)$, $E_1\in[1,\infty)$ and $\alpha\in(0,1)$, there exist $\Cl[eps]{e2}\in(0,\Cr{e1})$ and 
$\Cl[con]{c2}\in(1,\infty)$ depending only on $n,k,\alpha,
\nu,E_1$ with the following property. 
For $R\in(0,\infty)$, $T\in{\bf G}(n,k)$, and $U=\rC(T,2R)$,
suppose $\{V_t\}_{t\in[-R^2,0]}$ and $\{u(\cdot,t)\}_{t\in[-R^2,0]}$
satisfy (A1), (A2), (A4) and in place of (A3), assume $u\in C^{0,\alpha}(\rC(T,2R)\times[-R^2,0])$. Furthermore, assume \eqref{udef1},
\eqref{udef2}, \eqref{udef3} with $\Cr{e2}$, and in place of \eqref{udef4}, 
\begin{equation*}
   \|u\|_{\alpha}:=R \|u\|_0+R^{1+\alpha}[u]_{\alpha}
   <\Cr{e2}\,.
\end{equation*}
Then the conclusion
of Theorem \ref{main2} holds in the $C^{2,\alpha}$ class, that is \eqref{udef5} can be replaced by 
\begin{equation}
   R^{-1} \|f\|_0+\|\nabla f\|_0
   +R(\|\nabla^2 f\|_0+\|f_t\|_0)+R^{1+\alpha}
   ([\nabla^2 f]_\alpha+[f_t]_\alpha)
   \leq \Cr{c2}\max\{\mu,\|u\|_\alpha\}.
\end{equation}
Moreover, ${\rm image}\,F$ satisfies in the classical (pointwise) sense the
motion law that normal velocity $=h+u^\perp$.
\end{theorem}

Here one can extend $f$ and $F$ as $C^{2,\alpha}$
functions to $t=0$ on $B_{R/2} \cap T$. 
Once the regularity goes up to $C^{2,\alpha}$ and the
surfaces satisfy the PDE pointwise, then the parabolic Schauder
estimates can be applied in the case that $u$ is more regular. In particular, we will deduce $C^{k+2,\alpha}$ estimates if $u \in C^{k,\alpha}$. In the case of Brakke flow, when $u=0$, we have all the derivative
estimates in terms of $\mu$. 

In the next sections, we prove how the results stated in the Introduction, namely Theorem \ref{main1} and Theorem \ref{main1cor} can be deduced from Theorems \ref{main2} and \ref{main3}.

\subsection{Proof of Theorem \ref{main1}}\label{app1}

Let $E_0 \in \left( 0, \infty \right)$, and suppose $\{V_t\}_{t \in \left(-2,0\right]}$ is a $k$-dimensional Brakke flow satisfying (1)-(4) in Theorem \ref{main1} with $\varepsilon \in \left(0, \varepsilon_0\right]$. We prove that, if $\varepsilon_0$ is chosen sufficiently small, then $\{V_t\}_{t \in \left[-1,0\right]}$ satisfies the hypotheses of Theorem \ref{main3}. We set $R=1$, $T=\R^k\times\{0\}$, and $U=\rC(T,2)=:\rC_2$, and we notice that (A1)(A3)(A4) are satisfied by assumption. To check (A2), let $t \in \left[-1,0\right]$ and $B_r(x) \subset U$: it is then a classical consequence (see e.g. \cite[Proposition 3.5]{Ton1}) of Huisken's monotonicity formula that
\[
r^{-k} \|V_t\|(B_r(x)) \leq c\, \sup_{s\in \left[-2,t\right]} \|V_s\|(\rC_3) \leq c E_0\,,
\]
where $c$ is a universal constant. This proves (A2). Next, using that
\[
\phi_T^2 \leq \mathrm{1}_{\rC_1} \qquad \mbox{and} \qquad \mathbf{c} \geq \frac{2}{3} \omega_k\,,
\]
we see that (2) implies
\[
\|V_{-4/5}\|(\phi_T^2) \leq \|V_{-4/5}\|(\rC_1) \leq \frac54 \omega_k \leq \frac{15}{8} \, \mathbf{c}\,, 
\]
that is \eqref{udef1} holds with $\nu = 1/8$. Also, \eqref{udef3} with $\Cr{e2}$ is an immediate consequence of (4) as soon as $\varepsilon_0 \leq \Cr{e2}$, whereas \eqref{udef2} follows from (3) and Huisken's monotonicity formula (see, for instance, \cite[Proposition 3.6]{Ton1}). Hence, Theorem \ref{main3} applies, and Theorem \ref{main1} follows from the fact that the forcing field $u \equiv 0$ is smooth. \qed

\subsection{Proof of Theorem \ref{main1cor}} \label{app2}

In order to simplify the presentation, we will work under the assumption that $U = \R^n$, and that $\spt\|V_t\| \subset B_R$ for every $t \in \left(a, b \right]$, for some $R >0$. The general case can be obtained with simple modifications, but the underlying idea is the same; see Remark \ref{rmk:localization+manifold}.

Before proceeding with the proof, let us recall the classical definition of Gaussian density in the context of Brakke flows; see for instance \cite{White_stratification} for a thorough presentation. Under the above assumptions, and setting $\mathscr{V} = \{V_t\}_{t\in \left(a, b \right]}$, for any point $(x_0,t_0) \in \R^n \times \left( a, b \right]$ we define
\begin{equation} \label{def:density}
   \Theta(\mathscr{V},(x_0,t_0)) := \lim_{\tau \to 0^+} \frac{1}{(4\pi \tau)^{\frac{k}{2}}} \int_{\R^n}  \exp \left( - \frac{|y-x_0|^2}{4\tau} \right) \, d\|V_{t_0-\tau}\|(y)\,.
\end{equation}
The existence of the above limit is guaranteed by the fact that the function 
\[
\tau \in \left( 0, t_0-a \right) \mapsto \frac{1}{(4\pi \tau)^{\frac{k}{2}}} \int_{\R^n}  \exp \left( - \frac{|y-x_0|^2}{4\tau} \right) \, d\|V_{t_0-\tau}\|(y)
\]
is monotone increasing as a consequence of Huisken's monotonicity formula.

\smallskip

\textit{Step one.} Assume first that $\Theta (\mathscr{V},(x_0,t_0)) = 1$, and let $\mathscr{V}' = \{V_t'\}_{t \in \left(-\infty,0\right)}$ be any tangent flow to $\mathscr{V}$ at $(x_0,t_0)$. We then have
\[
1 = \Theta(\mathscr{V},(x_0,t_0)) = \Theta(\mathscr{V}',(0,0))\,,
\]
so that, in particular, $\Theta (\mathscr{V}', (y,s)) \leq 1$ for every $(y,s) \in \R^n \times \left(-\infty, 0 \right)$. Since it is a general fact that $\Theta (\mathscr{V},(y,s)) \geq 1$ for an integral Brakke flow $\mathscr{V}$ (see Appendix \ref{a:densities and tangent flows}), for every $(y,s) \in \spt(\|V_t'\|\times dt)$, we have
\[
\Theta (\mathscr{V}', (y,s)) = 1 = \Theta (\mathscr{V}', (0,0)) \qquad \mbox{for all $(y,s) \in \spt(\|V_t'\|\times dt)$}\,.
\]
This immediately implies (see e.g. \cite[Theorem 8.1]{White_stratification}) that there exist $a \in \left[0,\infty\right]$ and $T \in \mathbf{G}(n,k)$ such that
\[
V_t' = \mathbf{var}(T,1)\qquad \mbox{for every $t \in \left(-\infty, a\right)$}\,,
\]
namely that $\mathscr{V}'$ is a static $k$-dimensional plane with unit density. Therefore, there exists $\rho > 0$ such that the hypotheses of Theorem \ref{main3} are satisfied with $R=\rho$ by the flow $\{(\tau_{x_0})_\sharp V_{t_0+s}\}_{s \in \left[-\rho^2,0\right]}$, where $\tau_{x_0}$ is the translation $\tau_{x_0}(x) := x-x_0$. Thus, by Theorem \ref{main3}, for all $t \in \left[t_0-\rho^2/4,t_0\right)$, $\spt\|V_t\| \cap \left( x_0 + \rC(T,\rho/2)\right)$ coincides with the graph of a $C^\infty$ function
\[
f \colon B_{\rho/2}(x_0) \cap (x_0 + T) \times \left[t_0 - \rho^2/4,t_0\right) \to T^\perp
\]
which satisfies the mean curvature flow in the classical sense and which can be extended smoothly on $B_{\rho/2}(x_0) \cap (x_0+T)$ up to $t=t_0$. This completes the proof in case $\Theta(\mathscr{V},(x_0,t_0))=1$.

\smallskip

\textit{Step two.} The proof that the same result holds when $\Theta(\mathscr{V},(x_0,t_0)) \leq 1+\varepsilon_1$ for some sufficiently small $\varepsilon_1$ is by a standard blow-up argument. First, notice that it is sufficient to prove that there exists $\varepsilon_1>0$ such that if $\mathscr{V}$ is a tangent flow \footnote{That is, $\mathscr{V}=\{V_t\}_{t \in (-\infty,0)}$ satisfies $\|(\iota_r)_\sharp V_{-r^2}\| = \|V_{-1}\|$ for every $r > 0$, where $\iota_r(x):=r^{-1}x$, as well as 
\[
h(V_t,x)=\frac{S^\perp(x)}{2t} \qquad \mbox{for $V_t$-a.e. $(x,S) \in G_k(\R^n)$, for a.e. $t < 0$}\,.
\]
} and $\Theta(\mathscr{V},(0,0)) \leq 1 + \varepsilon_1$ then $\mathscr{V}$ is a static $k$-dimensional plane with unit density.

To see this, let $\{\mathscr{V}_j\}_{j \in \mathbb N}$ be a sequence of tangent flows such that $\Theta(\mathscr{V}_j, (0,0)) \leq 1 + \sfrac{1}{j}$, and notice that, for each $j$, the function
\[
\tau \in \left( 0, \infty \right) \mapsto \frac{1}{(4\pi\tau)^{\frac{k}{2}}} \int_{\R^n} \exp\left( -\frac{|y|^2}{4\tau} \right) \, d\|(V_j)_{-\tau}\|(y)
\]
is constant, so that, in particular
\[
\frac{1}{(4\pi)^{\frac{k}{2}}} \int_{\R^n} \exp\left( -\frac{|y|^2}{4} \right) \, d\|(V_j)_{-1}\|(y) = \Theta(\mathscr{V}_j,(0,0)) \leq 1 + \frac{1}{j}\,.
\]
Apply next the compactness theorem for Brakke flows, and let $\mathscr{V}$ be the limit Brakke flow of a (not relabeled) subsequence of $\{\mathscr{V}_j\}_j$. We have then 
\[
\begin{split}
1 \leq \Theta (\mathscr{V}, (0,0)) & \leq \frac{1}{(4\pi)^{\frac{k}{2}}} \int_{\R^n} \exp\left( -\frac{|y|^2}{4} \right) \, d\|V_{-1}\|(y) \\& \leq \liminf_{j \to \infty}\frac{1}{(4\pi)^{\frac{k}{2}}} \int_{\R^n} \exp\left( -\frac{|y|^2}{4} \right) \, d\|(V_j)_{-1}\|(y)\leq 1\,,
\end{split}
\]
and thus $\Theta(\mathscr{V},(0,0))=1$. By \textit{Step one}, $\spt\|V_t\|$ is a smooth graph evolving by mean curvature in some $B_\rho(0)$ for all $t \in \left[-\rho^2,0\right)$, and the flow can be extended smoothly in $B_\rho (0)$ up to $t=0$. Since $\mathscr{V}$ is the limit Brakke flow of $\mathscr{V}_j$, then for all sufficiently large $j$ also the flow $\mathscr{V}_j$ satisfies all assumptions of Theorem \ref{main3} (with $u \equiv 0$) in a parabolic domain $\rC(T,\rho/2) \times [-\rho^2/16,0]$, and thus $\mathscr{V}_j$ is a smooth mean curvature flow in a neighborhood of $x=0$ until the end-time $t=0$. Since $\mathscr{V}_j$ is a tangent flow, it must then be a static $k$-dimensional plane with unit density, and the proof is complete. \qed

\begin{remark}\label{rmk:localization+manifold}
In case $\{V_t\}_{t \in (a,b]}$ is a $k$-dimensional Brakke flow in a domain $U \subset \R^n$, the same proof goes through, except that we need a suitably modified monotonicity formula to make sense of the Gaussian density. More precisely, if $(x_0,t_0) \in U \times (a,b]$ and $B_{2r}(x_0) \subset U$ then for any function $\psi \colon B_{2r}(x_0) \to [0,1]$ that is smooth, compactly supported, equal to $1$ on $B_r(x_0)$ and satisfying a bound of the form $r |\nabla \psi| + r^2 \|D^2\psi\|\leq b$, the limit
\[
\lim_{\tau \to 0^+} \frac{1}{(4\pi \tau)^{\frac{k}{2}}} \int_{\R^n}  \exp \left( - \frac{|y-x_0|^2}{4\tau} \right) \, \psi (y) \, d\|V_{t_0-\tau}\|(y)
\]
exists and it is independent of $\psi$. This limit is the Gaussian density $\Theta(\mathscr{V}, (x_0,t_0))$. The same limit also exists in the case when $\mathscr{V}=\{V_t\}_{t \in (a,b]}$ is a $k$-dimensional Brakke flow in a domain $U$ of an $n$-dimensional Riemannian manifold $M$, or, more generally, when $\mathscr{V}=\{V_t\}_{t\in (a,b]}$ is a flow with a locally bounded forcing term $u$, with the only caveat that the proof of existence of the limit involves a more complicated monotonicity formula. Once the existence of the density has been established, tangent flows to such a $\mathscr{V}$ at $(x_0,t_0)$ are Brakke flows in $\R^n$ (in the manifold case, we are identifying $\R^n$ with ${\rm Tan}_{x_0}M$), and the proof proceeds verbatim. For the proof of the monotonicity formulas needed in these cases, the interested reader can consult \cite[Sections 10 and 11]{White_stratification}.
\end{remark}

\section{Energy estimates} \label{s:en_est}

The main result of this section is the following theorem, which establishes that the deviation of the $k$-dimensional area of surfaces that are $L^2$-close to a plane $T$ and move by (forced) unit-density Brakke flow from the area of a single $k$-dimensional disk can be estimated in terms of the $L^2$-height with respect to $T$. An analogous result was proved in \cite[Theorem 5.7]{Kasai-Tone}, but the version we are going to present here has an important advantage, which is ultimately the key to unlock the end-time regularity. More precisely, while \cite[Theorem 5.7]{Kasai-Tone} concludes the validity of the estimate up to some waiting time both at the beginning and at the end of the time interval where the $L^2$-height is assumed to be small, here we extend the estimate arbitrarily near the end-time as long as we know that the area of the moving surfaces is a sufficiently large portion of the area of the disk (namely, as long as we know that the flow is not vanishing). The price to pay is that the estimate comes with a constant which deteriorates while approaching the end-time. The end-time regularity will result from appropriately balancing the size of this constant with the vanishing of the $L^2$-height along a blow-up sequence. 

\begin{theorem} \label{t:mass_growth_control}
Corresponding to $E_1\in[1,\infty)$ and $\tau\in \left(0,\frac12\right)$, there exist $\Cl[eps]{e_en}=\Cr{e_en}(E_1,\tau) \in (0,1)$ and $K = K(E_1) \in \left(1,\infty\right)$ independent of $\tau$ with the following property. Given $T \in \mathbf{G}(n,k)$, suppose that $\{V_t\}_{t \in \left[-1,0\right]}$ and $\{u(\cdot,t)\}_{t \in \left[-1,0\right]}$ satisfy (A1)-(A4) with $U = \mathrm{C}(T,1)$. Assume also that
\begin{eqnarray}
& \exists\, C > 0 \, \colon \, \spt\|V_t\| \subset \mathrm{C}(T,1) \cap \{|T^\perp(x)| < C\} \qquad \forall\,t\in \left(-1,0\right]\,; \label{hp:Linfty_bound} \\
&\mu_*^2 := \displaystyle\sup_{t \in \left[-1,0\right]} \int_{\mathrm{C}(T,1)} |T^\perp(x)|^2\, d\|V_t\|(x) \leq \Cr{e_en}^2\,; \label{hp:excess_uniformly_small} \\
&  \|V_{-1}\|(\phi_T^2) - \bc  \leq \Cr{e_en}^2\,; \label{hp:in_cl_1disk} \\
&C(u):= \displaystyle\int_{-1}^0\int_{\rC(T,1)} 2\, |u|^2\,\phi_T^2\, d\|V_t\|dt \leq \Cr{e_en}^2\,. \label{hp:small_transport}
\end{eqnarray}
Then,
\begin{equation} \label{e:mass_growth_control}
    \sup_{t\in \left[-\frac12,0\right]} \|V_t\|(\phi_T^2) \leq \bc + K (\mu_*^2 + C(u))\,.
\end{equation}

Furthermore, if 
\begin{equation} \label{hp:tau}
    \sup_{t\in \left[-\tau,0\right]} \|V_t\|(\phi_T^2) - \bc  \geq - \Cr{e_en}^2\,,
\end{equation}

then

\begin{equation} \label{e:energy_estimate}
    \sup_{t\in\left[-\frac12,-2\tau\right]} \left| \|V_t\|(\phi_T^2) - \bc \right| \leq \frac{K}{\tau^3} (\mu_*^2 + C(u))\,.
\end{equation}

\end{theorem}

Before coming to the proof of Theorem \ref{t:mass_growth_control}, we record here the following result, which is \cite[Proposition 5.2]{Kasai-Tone}. 
\begin{proposition} \label{p:cheap_energy_est}
Corresponding to $E_1\in[1,\infty)$ and $\nu \in(0,1)$ there exist $\alpha_2\in(0,1)$, $\mu_1\in(0,1)$, and $P_2\in[1,\infty)$ with the following property. For $T \in \bG(n,k)$ and a unit density varifold $V \in \mathbf{IV}_k (\rC(T,1))$ with finite mass, define
\begin{align}\label{def:mean_curvature}
    \alpha^2 &:= \int_{\rC(T,1)} |h(V,x)|^2\,\phi_T^2(x) \, d\|V\|(x)\,, \\ \label{def:height}
    \mu^2 &:= \int_{\rC(T,1)} |T^\perp(x)|^2\,d\|V\|(x)\,.
\end{align}

Suppose $\spt\|V\|$ is bounded and 
\begin{equation} \label{denrat_prop}
    \|V\|(B_r(x)) \leq \omega_k r^k E_1 \quad \mbox{for all $B_r(x) \subset \rC(T,1)$}\,.
\end{equation}

\begin{itemize}

\item[(A)] If
\begin{equation} \label{case(A)}
    \left| \|V\|(\phi_T^2) - \bc \right| \leq \frac{\bc}{8}\,, \quad \alpha \leq \alpha_2\,, \quad \mbox{and $\mu \leq \mu_1$}\,,
\end{equation}
then we have
\begin{equation} \label{e:cheap_energy_est}
     \left| \|V\|(\phi_T^2) - \bc \right| \leq
     \begin{cases}
     P_2(\alpha^{\frac{2k}{k-2}} + \alpha^{\frac32} \mu^{\frac12} + \mu^2) &\mbox{if $k > 2$}\,, \\
     P_2 (\alpha^{\frac32} \mu^{\frac12} + \mu^2) &\mbox{if $k \leq 2$}\,.
     \end{cases}
  \end{equation}
   
\item[(B)] If, instead
\begin{equation} \label{case(B)}
    \frac{\bc}{8} < \left| \|V\|(\phi_T^2) - \bc \right| \leq (1-\nu) \bc \quad \mbox{and $\mu \leq \mu_1$}
\end{equation}
then $\alpha \geq \alpha_2$.
\end{itemize}
\end{proposition}

The following is an immediate corollary of Proposition \ref{p:cheap_energy_est}, and it is \cite[Corollary 5.3]{Kasai-Tone}

\begin{corollary} \label{cor:cheap_energy_estimate}
Let $\alpha_2,\mu_1$, and $P_2$ be as in Proposition \ref{p:cheap_energy_est}. Set $\mu_2 := \min\{\mu_1, \left( \frac{\bc}{32P_2} \right)^{1/2}\}$. For $V$ and $T$ as in Proposition \ref{p:cheap_energy_est}, define $\alpha$ and $\mu$ as in \eqref{def:mean_curvature} and \eqref{def:height}. Also define
\begin{equation} \label{def:mass_excess}
    \hat E := \|V\|(\phi_T^2) - \bc\,.
\end{equation}
Assume \eqref{denrat_prop} as well as
\begin{equation} \label{ass_cor_cheap}
    \mu \leq \mu_2\,, \quad \mbox{and} \quad 2P_2\mu^2 \leq |\hat E| \leq (1-\nu)\bc\,.
\end{equation}
Then, we have
\begin{equation} \label{conclusion:large dissipation}
    \alpha^2 \geq 
    \begin{cases}
    \min\left\lbrace \alpha_2^2, (4P_2)^{-\frac{k-2}{k}}|\hat E|^{\frac{k-2}{k}}, (4P_2)^{-\frac43}\mu^{-\frac23} |\hat E|^{\frac43} \right\rbrace & \mbox{if $k > 2$}\,, \\
    \min\left\lbrace \alpha_2^2, (2P_2)^{-\frac43} \mu^{-\frac23} |\hat E|^{\frac43} \right\rbrace & \mbox{if $k \leq 2$}\,.
    \end{cases}
\end{equation}
\end{corollary}

\begin{proof}[Proof of Theorem \ref{t:mass_growth_control}]

The general scheme follows the proof of \cite[Theorem 5.7]{Kasai-Tone}. We define the function
\begin{equation} \label{e:monotone_energy}
    t \in \left[-1,0\right] \mapsto E(t) := \|V_t\|(\phi_T^2) - \bc - \int_{-1}^t \int_{\rC(T,1)} 2|u|^2\phi_T^2\, d\|V_s\|ds - K_2\mu_*^2 (1+t)\,,
\end{equation}
where 
\[
K_2 := 80\,\sup\{5|\nabla\phi_T|^4 \phi_T^{-2} + |\nabla|\nabla \phi_T||^2\}\,.
\]
Arguing precisely as in the proof of \cite[(5.53)]{Kasai-Tone}, namely by testing Brakke's inequality \eqref{Bra-ineq} with $\varphi = \phi_T^2$, we conclude that
\begin{equation} \label{e:energy_monotonicity}
 E(t_2) - E(t_1) \leq - \frac14 \int_{t_1}^{t_2} \int_{\rC(T,1)} |h(V_t,\cdot)|^2 \, \phi_T^2 \,d\|V_t\|dt \qquad   \mbox{for every $-1 \leq t_1 < t_2 \leq 0$}\,.
\end{equation}
We first prove \eqref{e:mass_growth_control}. Towards a contradiction, suppose that there exists $t_* \in \left[-\frac12,0\right]$ such that 
\begin{equation} \label{hp:contradiction}
    \|V_{t_*}\|(\phi_T^2) - \bc > K (\mu_*^2 + C(u))\,,
\end{equation}
where $1 < K < \infty$ will be chosen later. In particular, from the definition of $E(t)$ we have for every $t \in \left[-1,t_*\right]$ that
\begin{equation} \label{hp:contradiction2}
\|V_t\|(\phi_T^2) - \bc \geq E(t) \overset{\eqref{e:energy_monotonicity}}{\geq} E(t_*) > K(\mu_*^2 + C(u)) - C(u) - K_2\mu_*^2 \geq \frac{K}{2}\mu_*^2 
\end{equation}
if we choose $K \ge \max\{1, 2 K_2\}$. On the other hand, we also have, due to \eqref{e:energy_monotonicity}, \eqref{hp:excess_uniformly_small}, \eqref{hp:in_cl_1disk}, and \eqref{hp:small_transport},
\[
\|V_t\|(\phi_T^2) - \bc \leq E(t) + C(u) + K_2\mu_*^2 \leq E(-1) + C(u) + K_2\mu_*^2 \leq (K_2+2)\,\Cr{e_en}^2 \leq \Cr{e_en} \bc
\]
for $\Cr{e_en}$ suitably small. In particular, if $P_2$ is the constant from Proposition \ref{p:cheap_energy_est} corresponding to $E_1$ and, for instance, $\nu=1/2$, then choosing also $K \ge 4P_2$ we have that
\begin{equation} \label{energy bounds}
    2P_2 \mu_*^2 \leq \|V_t\|(\phi_T^2) -\bc \leq \Cr{e_en} \bc \qquad \mbox{for every $t \in \left[-1,t_*\right]$}\,.
\end{equation}
Hence, we can apply Corollary \ref{cor:cheap_energy_estimate} with $V=V_t$ for all $t \in \left[-1,t_*\right]$, and conclude that for a.e. $t\in\left[-1,t_*\right]$ it holds
\begin{equation} \label{lb_mc}
    \frac14 \,\int_{\rC(T,1)} |h(V_t,\cdot)|^2\phi_T^2\,d\|V_t\| \geq 
    \begin{cases}
    P \min\{1,E(t)^{\frac{k-2}{k}}, \mu_*^{-\frac23}E(t)^{\frac{4}{3}}\} &\mbox{if $k>2$}, \\
    P \min\{1,\mu_*^{-\frac23}E(t)^{\frac{4}{3}}\} &\mbox{if $k\leq 2$}\,,
    \end{cases}
\end{equation}
where
\[
P := \frac{1}{4 \cdot 2^{4/3}} \min\{\alpha_2^2, (4P_2)^{-\frac{k-2}{k}}, (4P_2)^{-\frac43}\}\,,
\]
and $\alpha_2 \in \left(0,1\right)$ is the same constant as in Proposition \ref{p:cheap_energy_est} corresponding to $E_1$ and $\nu = 1/2$. Let us consider the case $k > 2$, as the case $k \leq 2$ is easier and can be treated similarly. Note that, since $\Cr{e_en} < 1$,
\[
P\min\{1,E(t)^{\frac{k-2}{k}}, \mu_*^{-\frac23}E(t)^{\frac{4}{3}}\} = 
\begin{cases}
P & \mbox{if $E(t) \geq 1$},\\
P E(t)^{\frac{k-2}{k}} &\mbox{if $\mu_*^{\frac{2k}{k+6}} \leq E(t) \leq 1$}, \\
P \mu_*^{-\frac23} E(t)^{\frac43} & \mbox{if $E(t) \leq \mu_*^{\frac{2k}{k+6}}$}\,.
\end{cases}
\]
On the other hand, for $t \in \left[-1,t_*\right]$ we have
\[
E(t) \leq E(-1) = \|V_{-1}\|(\phi_T^2) - \bc \leq \Cr{e_en}^2 < 1\,,
\]
so that the first alternative does not occur. Let $\bar t$ be the supremum of $s \in [-1,t_*]$ such that $\mu_*^{\frac{2k}{k+6}} \leq E(t) \leq 1$ for $t \in \left[-1, s \right]$. Then, \eqref{e:energy_monotonicity} and \eqref{lb_mc} imply that the differential inequality $E'(t) \leq - P E(t)^{\frac{k-2}{k}}$ is satisfied a.e. on $\left[-1, \bar t \right]$. Integrating and using \eqref{hp:in_cl_1disk}, we find then that 
\[
\bar t \leq -1 + \frac{k \Cr{e_en}^{\frac{4}{k}}}{2P}\,.
\]
In particular, for $\Cr{e_en}$ suitably small it is $\bar t < -\frac34$. By the monotonicity of $E(t)$, we then have that the differential inequality $E'(t) \leq - P \mu_*^{-\frac23} E(t)^{\frac43}$ is satisfied a.e. on $\left[ \bar t, t_* \right]$, so that, integrating, we find
\begin{equation} \label{diff_ineq_estimate}
E(t_*) \leq \left( \frac{3}{P (t_* - \bar t)} \right)^3 \mu_*^2\,.
\end{equation}
Since $t_*-\bar t \geq 1/4$, \eqref{diff_ineq_estimate} is in contradiction with \eqref{hp:contradiction2} as soon as we choose $K \geq \frac{4}{P^3} 12^3$. This completes the proof of \eqref{e:mass_growth_control}. Assume now that \eqref{hp:tau} holds, and let $\bar t \in \left[-\tau,0\right]$ be such that
\begin{equation} \label{hp:tau2}
    \|V_{\bar t}\|(\phi_T^2) - \bc \geq -\frac32 \Cr{e_en}^2\,.
\end{equation}
Towards a contradiction, assume that \eqref{e:energy_estimate} is violated: due to \eqref{e:mass_growth_control}, this means that there exists $t_* \in \left[-\frac12,-2\tau\right]$ such that 
\begin{equation}  \label{hp:contradiction_lb}
E(t_*) \leq \|V_{t_*}\|(\phi_T^2) - \bc < - \frac{K}{\tau^3}(\mu_*^2+C(u))\,.
\end{equation}
We then have 
\begin{equation} \label{hp:contradiction_lb_2}
E(t) \leq - \frac{K}{\tau^3}(\mu_*^2+C(u)) \qquad \mbox{for every $t \in \left[t_*,\bar t \right]$}
\end{equation}
by monotonicity, and thus
\[
\|V_t\|(\phi_T^2)-\bc \leq E(t) + C(u) + K_2\mu_*^2 \leq - \frac{K}{2} \mu_*^2 \qquad \mbox{for every $t \in \left[t_*,\bar t \right]$}\,.
\]
On the other hand, again for $t \in \left[t_*,\bar t \right]$ we have
\[
\|V_t\|(\phi_T^2) - \bc \ge E(t) \ge E(\bar t) \geq - \left(\frac{5}{2} + K_2\right) \Cr{e_en}^2 \geq -\Cr{e_en} \bc\,,
\]
for $\Cr{e_en}$ sufficiently small, where we have used \eqref{hp:tau2} together with \eqref{hp:excess_uniformly_small} and \eqref{hp:small_transport}. We can then apply again Corollary \ref{cor:cheap_energy_estimate} with $V=V_t$, $t \in \left[t_*,\bar t\right]$, and conclude that for a.e. $t \in \left[t_*,\bar t \right]$
\begin{equation} \label{lb_mc_2_pre}
\begin{split}
    &\frac14 \,\int_{\rC(T,1)} |h(V_t,\cdot)|^2\phi_T^2\,d\|V_t\| \\ & \qquad \geq 
    \begin{cases}
    2^{\frac43} P \min\{1,\left(\bc - \|V_t\|(\phi_T^2)\right)^{\frac{k-2}{k}}, \mu_*^{-\frac23}\left(\bc - \|V_t\|(\phi_T^2)\right)^{\frac{4}{3}}\} &\mbox{if $k>2$}, \\
    2^{\frac43} P \min\{1,\mu_*^{-\frac23}\left(\bc - \|V_t\|(\phi_T^2)\right)^{\frac{4}{3}}\} &\mbox{if $k\leq 2$}\,.
    \end{cases}
\end{split}
\end{equation}

On the other hand, as a consequence of \eqref{hp:contradiction_lb_2} we have that for every $t \in \left[t_*,\bar t \right]$
\[
\bc - \|V_t\|(\phi_T^2) \geq -E(t) - C(u) - K_2 \mu_*^2 \geq -E(t) - K (C(u) + \mu_*^2) \geq (-1+\tau^3) E(t) \geq \frac12 (-E(t))\,,
\]
and thus
\begin{equation} \label{lb_mc_2}
    \frac14 \,\int_{\rC(T,1)} |h(V_t,\cdot)|^2\phi_T^2\,d\|V_t\| \geq 
    \begin{cases}
    P \min\{1,\left(-E(t)\right)^{\frac{k-2}{k}}, \mu_*^{-\frac23}\left(-E(t)\right)^{\frac{4}{3}}\} &\mbox{if $k>2$}, \\
    P \min\{1,\mu_*^{-\frac23}\left(-E(t)\right)^{\frac{4}{3}}\} &\mbox{if $k\leq 2$}\,.
    \end{cases}
\end{equation}
Arguing as above, we only treat the case $k>2$, and we notice that $-E(t)=|E(t)| < 1$. Assume that $\hat t$ is the infimum of $s \in \left[t_*, \bar t \right]$ such that $|E(t)| \geq \mu_*^{\frac{2k}{k+6}}$ for all $t \in \left[s, \bar t\right]$. Then, \eqref{e:energy_monotonicity} and \eqref{lb_mc_2} imply that the differential inequality $E'(t) \leq - P \left(- E(t) \right)^{\frac{k-2}{k}}$ is satisfied a.e. on $\left[\hat t, \bar t \right]$. Integrating we find that 
\[
\frac{2P}{k} (\bar t-\hat t) \leq \left(-E(\bar t)\right)^{\frac{2}{k}} - \left(-E(\hat t)\right)^{\frac{2}{k}} \leq (\Cr{e_en} \bc)^{\frac{2}{k}}\,.
\]
In particular, for $\Cr{e_en}$ sufficiently small (depending on $\tau$) we have $\hat t \in \left[-\frac{3}{2}\tau ,\bar t\right]$. Now, by monotonicity of $E(t)$, it holds $|E(t)| \leq \mu_*^{\frac{2k}{k+6}}$ on $\left[t_*,\hat t \right]$, and thus the differential inequality $E'(t) \leq - P \mu_*^{-\frac23} \left(-E(t)\right)^{\frac43}$ holds a.e. on $\left[t_*,\hat t \right]$. We integrate to find that 
\[
E(t_*) \geq - \left( \frac{3}{P (\hat t - t_*)} \right)^{3} \mu_*^2 \geq - \left( \frac{6}{P \tau} \right)^3 \mu_*^2\,,
\]
which contradicts \eqref{hp:contradiction_lb} if $K \geq 2 (6/P)^3$ and completes the proof of \eqref{e:energy_estimate}.
\end{proof}

%%%%%%%%%%%%%%%%%%%%%%%%%%%%%%%%%%%%%%%%%%%%%%%%%%
\section{Lipschitz approximation} \label{s:Lip_approx}

The following proposition states the existence of a Lipschitz approximation of the flow in space-time, with good estimates up to the end-time. The result is similar
to \cite[Theorem 7.5]{Kasai-Tone}, 
the only difference being that 
the Lipschitz approximation is obtained
up to the end-time. In the next Section
\ref{buarg}, 
$t=0$ in Proposition \ref{lipap}
will correspond to a time slightly before
the end-time, up to which we have a good 
excess estimate. 

\begin{proposition}\label{lipap}
Corresponding to $E_1\in[1,\infty)$, $p$ and $q$, there exist $\Cl[eps]{e3}\in(0,1)$, $r_1\in(0,1)$ and $\Cl[con]{c3}\in[1,\infty)$ with the following property. For $U=\rC(T,1)$, suppose that $\{V_t\}_{t\in[-3/5,0]}$ and $\{u(\cdot,t)\}_{t\in[-3/5,0]}$ satisfy (A1)-(A4). 
Write $V_t={\bf var}(M_t,1)$ for a.e.\,$t$ and identify $T$ with $\mathbb R^k\times\{0\}$.
Suppose that we have
\begin{equation}\label{lipap1}
    \int_{\rC(T,1)\times[-3/5,0]} |h(V_t,\cdot)|^2\phi_T^2\,d\|V_t\|dt
    \leq \Cr{e3} r_1^2/4\,,
\end{equation}
\begin{equation}\label{lipap2}
    \big|\|V_t\|(\phi_T^2)-{\bf c}\big|\leq \Cr{e3} 
    \,\,\,\,\mbox{ for all }t\in[-3/5,0]\,,
\end{equation}
\begin{equation}\label{lipap3}
    {\rm spt} \,\|V_t\|\cap \rC(T,1)\subset \{|T^{\perp}(x)|\leq
    \Cr{e3}\}\,\,\,\mbox{ for all }t\in [-3/5,0]\,,
\end{equation}
\begin{equation}\label{lipap4}
    \|u\|_{L^{p,q}(\rC(T,1)\times[-3/5,0])}\leq 1\,.
\end{equation}
Set 
\begin{equation}\label{lipap5}
    \beta^2:=\int_{G_k(\rC(T,1))\times[-3/5,0]} \|S-T\|^2\phi_T^2\,
    dV_t(\cdot,S)dt
\end{equation}
and
\begin{equation}\label{lipap6}
    \kappa^2:=\left|\int_{-3/5}^0\Big(\|V_t\|(\phi_{T,1/2}^2)-
    \frac{\bf c}{2^k}\Big)\,dt
    \right|\,.
\end{equation}
Then there exist maps $f\,:\, B_{1/3}^k\times [-1/2,0]\rightarrow
\mathbb R^{n-k}$ and $F\,:\,B_{1/3}^k\times[-1/2,0]\rightarrow
\mathbb R^n\times[-1/2,0]$ such that for all $(x,s),\,(y,t)\in B_{1/3}^k
\times[-1/2,0]$, 
\begin{equation}\label{lipap7}
\begin{split}
    &F(x,s)=(x,f(x,s),s)\,, \\
    &|f(x,s)-f(y,t)|\leq c(n,k)\max\{|x-y|,|s-t|^{1/2}\}\,, \\
    &|f(x,s)|\leq \Cr{e3}\,,
\end{split}
\end{equation}
and with the following property. Define 
\begin{equation}\label{lipap8}
\begin{split}&X:=\Big(\cup_{t\in[-1/2,0]} (M_t\cap \rC(T,1/3))\times\{t\}\Big)
\cap {\rm image}\,F\,, \\
&Y:=(T\times {\rm Id}_{\mathbb R} )(X)\,.
\end{split}
\end{equation}
Then
\begin{equation}\label{lipap9}
        (\|V_t\|\times dt)((\rC(T,1/3)\times[-1/2,0])\setminus X) 
    +\mathcal L^{k+1}((B_{1/3}^k\times[-1/2,0])\setminus Y)
    \leq \kappa^2+\Cr{c3}\beta^2\,.
\end{equation}
\end{proposition}

\begin{proof} To be consistent with the notation in \cite[Section 7]{Kasai-Tone},
we change the time intervals $[-3/5,0]$ and $[-1/2,0]$
in the statement above to $[0,1]$ and $[1/4,1]$ respectively in the following, 
which does not change the proof in any essential way. With this
replacement, we discuss the proof. 
We simply describe the exact locations 
where we need to change in \cite[Section 7]{Kasai-Tone} and the equation numbers are those of \cite{Kasai-Tone} in the following for the rest of the proof. 
For \cite[Proposition 7.1]{Kasai-Tone},
one replaces
the 
parabolic cylinder $P_r(a,s)$ in (7.3) and (7.4)
by $\tilde P_r(a,s)$ defined in Section \ref{notation}
and the same conclusion (7.6)
follows by the same proof. 
Next, no change is required in \cite[Lemma 7.3]{Kasai-Tone},
where one obtains a small constant $r_1\in(0,1)$ depending
only on $E_1,\,p$ and $q$. In the proof of \cite[Theorem 7.5]{Kasai-Tone}, one replaces $(1/4,3/4)$ by $(1/4,1)$ and $P$ by $\tilde P$
in (7.58), (7.59), (7.62), (7.65) and (7.66). 
The only essential modification is
the part following (7.66) on the covering argument.
The modified statement (7.66) is the following:
For each $(x,s)\in B$, there exists some $r(x,s)\in(0,r_1)$ such 
that 
\begin{equation*}
    \int_{\overline{\tilde P_{r(x,s)}(x,s)}}\|S-T\|^2\,dV_t(\cdot,S)dt\geq \gamma
    (r(x,s))^{k+2}.
\end{equation*}
This follows from the definition of $A$, (7.58).
Thus $\{\overline{\tilde P_{r(x,s)}(x,s)}\}_{(x,s)\in B}$ is a covering of $B$. Here, unlike $P_r(x,s)$, since
$\tilde P_{r}(x,s)$ is not a metric ball with respect to
the metric $d((x_1,s_1),(x_2,s_2)):=\max\{|x_1-x_2|,|s_1-s_2|^{1/2}\}$, we cannot invoke
the standard Vitali covering lemma as given. 
On the other hand, by
following the same proof of the Vitali lemma 
applied to $\{\overline{\tilde P_{r(x,s)}(x,s)}\}_{(x,s)\in B}$ (see for example \cite[Theorem 3.3]{Simon}), one can prove
that there exists a countable subset
$\{\overline{\tilde P_{r(x_j,s_j)}(x_j,s_j)}\}\subset
\{\overline{\tilde P_{r(x,s)}(x,s)}\}_{(x,s)\in B}$
such that it is pairwise disjoint and 
\begin{equation*}
    B\subset \cup_{(x,s)\in B} \overline{\tilde P_{r(x,s)}(x,s)}
    \subset \cup_{j} (\mathbb R^n\times(0,1])\cap \overline{P_{5r(x_j,s_j)}(x_j,s_j)}.
\end{equation*}
Note that the right-hand side
are the closed metric balls with respect to the parabolic distance. Then, using the above inequality and the property
of the covering, 
\begin{equation*}
\begin{split}
    (\|V_t\|\times dt)(B)&\leq \sum_j (\|V_t\|\times dt)
    ((\mathbb R^n\times(0,1]\cap \overline{P_{5r(x_j,s_j)}(x_j,s_j)}\,) \\
    &\leq \sum_j 5^{k+2}2 E_1 r(x_j,s_j)^{k+2} \\
    &\leq \sum_j 5^{k+2}2E_1\gamma^{-1}\int_{\overline{{\tilde P}_{r(x_j,s_j)}(x_j,s_j)}} \|S-T\|^2\, dV_t(\cdot,t)dt \\
    &\leq 5^{k+2}2E_1\gamma^{-1}\int_{\rC(T,13/24)\times(0,1)}
    \|S-T\|^2\,dV_t(\cdot,S)dt\leq 5^{k+2}2\gamma^{-1}\beta^2.
    \end{split}
\end{equation*}
The rest of the proof is the same. 
\end{proof}
\begin{remark}
In \cite{Kasai-Tone}, the generalized Besicovitch covering theorem in \cite[2.8.14]{Federer_book} was invoked for parabolic cylinders at 
the bottom of page 40. After the publication
of \cite{Kasai-Tone}, Ulrich Menne communicated the second-named author that
the parabolic cylinders do not satisfy the assumption in \cite[2.8.14]{Federer_book} (called directionally $\xi,\,\eta,\,\zeta$ limited), so that the theorem is not 
applicable. However, one can fix the proof in \cite{Kasai-Tone} by 
using the Vitali covering lemma, which holds true for any metric balls, 
instead of using Besicovich. Later it was proved that, even though the precise assumption in \cite{Federer_book} is not satisfied, the Besicovich covering theorem 
still holds true for parabolic cylinders of type
$P$ (not $\tilde P$), see \cite{Itoh} for the proof.
\end{remark}
%%%%%%%%%%%%%%%%%%%%%%%%%%%%%%%%%%%%%%%%%%%%%
\section{Blow-up argument}
\label{buarg}
We first state the regularity result for 
a domain which is at positive
distance away from the
end-time $t=0$.
This is a direct consequence of \cite[Theorem 8.7]{Kasai-Tone} with modifications to shorten the
waiting time near the end-time.
\begin{proposition}\label{blow-es}
Corresponding to $E_1\in[1,\infty)$, $\nu\in(0,1)$, $p$, $q$ and 
$\iota\in(0,1/4)$, there exist
$\Cl[eps]{e4}\in(0,1)$, $\Cl[con]{c4}\in(1,\infty)$ with the following property. For $T\in{\bf G}(n,k)$, $R\in(0,\infty)$, $U=\rC(T,2R)$, suppose
$\{V_t\}_{t\in[-R^2,0]}$ and $\{u(\cdot,t)\}_{t\in[-R^2,0]}$ 
satisfy (A1)-(A4)
and \eqref{udef1}-\eqref{udef4} with
$\Cr{e4}$ in place of $\Cr{e1}$.
Write $\tilde D:=(B_R \cap T)\times[-R^2/2,-\iota R^2]$.
Then there are $f\,:\,\tilde D\rightarrow
T^\perp$ and $F\,:\,\tilde D\rightarrow \mathbb R^n$
such that $T(F(y,t))=y$ and $T^{\perp}(F(y,t))=f(y,t)$
for all $(y,t)\in\tilde D$ and 
\begin{equation}\label{f5}
    {\rm spt}\,\|V_t\|\cap \rC(T,R)={\rm image}\,F(\cdot,t)
    \,\,\mbox{ for all }t\in [-R^2/2,-\iota R^2],
\end{equation}
\begin{equation}\label{f6}
   R^{-1} \|f\|_0+\|\nabla f\|_0+R^\zeta[ f]_{1+\zeta}\leq \Cr{c4}(\mu+\|u\|),
\end{equation}
where the norms are measured on $(B_R \cap T)\times[-R^2/2,-\iota R^2]$.
\end{proposition}
\begin{proof}
We may assume that $R=1$ by the parabolic change of variables.
We first use the $L^2-L^\infty$ height estimate \cite[Proposition 6.4]{Kasai-Tone} with 
$R=1$, $\Lambda=1$, $U=B_1(a)$ with $a\in T\cap B_1$ (and the time-interval $[0,1]$ translated to $[-1,0]$), so that 
there exist $\Cl[con]{c5}=\Cr{c5}(k,p,q)$ and $\Cl[con]{c6}=\Cr{c6}(n,k)$ such 
that, for all $t\in[-4/5,0]$, we have
\begin{equation}\label{non3}
    {\rm spt}\,\|V_t\|\cap B_{4/5}(a)
    \subset \{x\,:\,|T^{\perp}(x)|\leq \tilde\mu\},
\end{equation}
where
\begin{equation}\label{nonsup2}
    \tilde\mu^2:=\Cr{c6}\mu^2+3\Cr{c5}
   \|u\|^2 E_1^{1-\sfrac{2}{p}}.
\end{equation}
In particular, by moving $a$ within $T\cap B_1$, \eqref{non3} shows
\begin{equation}\label{non4}
    {\rm spt}\,\|V_t\|\cap \rC(T,3/2)\cap \{x\,:\,|T^{\perp}(x)|\leq 1/2\}\subset \{x\,:\,|T^\perp(x)|\leq \tilde \mu\}
\end{equation}
for all $t\in[-4/5,0]$. Using the lower density
ratio bound (see \cite[Corollary 6.3]{Kasai-Tone}), for
all sufficiently small $\Cr{e4}$ depending only on $E_1$, $p$ and $q$, one can show that
\begin{equation}\label{non5}
    {\rm spt}\,\|V_t\|\cap \rC(T,3/2)\cap\{x\,:\,
    |T^\perp(x)|>1/2\}=\emptyset
\end{equation}
for all $t\in[-4/5,0]$. 
Thus, \eqref{non4} and \eqref{non5} show
\begin{equation}\label{nonsup}
   {\rm spt}\,\|V_t\|\cap \rC(T,3/2)
   \subset \{x\,:\,|T^\perp(x)|\leq \tilde \mu\}
\end{equation}
for all $t\in[-4/5,0]$.
Next, we use \cite[Theorem 8.7]{Kasai-Tone}. 
Corresponding to
$E_1$, $p$ and $q$ with $\nu=1/2$, there exist $\Cl[eps]{e5}\in(0,1)$
($\varepsilon_6$ in \cite{Kasai-Tone}),
$\sigma\in (0,1/2)$, $\Cl[Lam]{lam1}\in(2,\infty)$ ($\Lambda_3$ in
\cite{Kasai-Tone}) and $\Cl[con]{c7}\in(1,\infty)$ ($c_{16}$
in \cite{Kasai-Tone})
with the properties stated there. We identify $T$
with $\mathbb R^k\times\{0\}$ in the following. We fix a small $0<\tilde R\leq 1/6$ depending only on $\iota$ and $\Cr{lam1}$ 
(for example, $\tilde R=\sqrt{\iota/(4\Cr{lam1})}$)
so that, for any $(x,t)\in 
B^k_1\times[-1/2,-\iota]$, we have
\begin{equation}
    B^k_{3\tilde R}(x)\times(t-\Cr{lam1}\tilde R^2,t+\Cr{lam1}\tilde R^2)\subset
B^k_{3/2}\times(-3/5,-\iota/2).
\end{equation} 
The choice of such $\tilde R$ depends 
ultimately only on $\iota$, $E_1$, $p$ and $q$. We use 
\cite[Theorem 8.7]{Kasai-Tone} with $R=\tilde R$ and $(x,t)\in B_1^k\times[-1/2,-\iota]$ as the origin. There are four assumptions in \cite[Theorem 8.7]{Kasai-Tone}, the smallness of 
height \cite[(8.83)]{Kasai-Tone} and $\|u\|$ \cite[(8.84)]{Kasai-Tone}, and 
the existence of $t_1$ and $t_2$ in \cite[(8.85)]{Kasai-Tone} and \cite[(8.86)]{Kasai-Tone}
with respect to $B^k_{3\tilde R}(x)\times(t-\Cr{lam1}\tilde R^2,t+\Cr{lam1}\tilde R^2)$ and $\nu=1/2$. 
The first two conditions are fulfilled if we restrict
$\Cr{e4}$ so that $\Cr{e4} \tilde R^{-(k+4)/2}<\Cr{e5}$. In the following, we prove that the latter two are satisfied by using a compactness 
argument. 
Let $\phi_{T,\tilde R,x}$ be defined
by $\phi_{T,\tilde R,x}(y):=\phi_{T,\tilde R}(y-x)$. We 
claim that, given any $\delta>0$, for all sufficiently 
small $\Cr{e4}>0$ depending only on 
$\iota,\,E_1,\,\nu,\,p,\,q$ and $\delta$, we have
\begin{equation}\label{ab3}
   \tilde R^{-k} \|V_t\|(\phi_{T,\tilde R,x}^2)\leq {\bf c}+\delta
\end{equation}
for all $(x,t)\in B^k_1\times [-3/5,0]$. Note that, by 
using the monotone decreasing property of $E(t)$ corresponding to 
$\phi_{T,\tilde R,x}$ in place of $\phi_T$ in \eqref{e:energy_monotonicity},
the increase of $\|V_t\|(\phi_{T,\tilde R,x}^2)$
in time can be made small by restricting $\mu$ and
$\|u\|$ appropriately depending on $\delta$ and $\tilde R$
(in the following, we may refer to this fact as ``almost monotone 
property''),
so we need to prove $\tilde R^{-k}\|V_{-3/5}\|(\phi_{T,\tilde R,x}^2)\leq {\bf c}+\delta$ for all $x\in B^k_1$.
Assume for a contradiction that there exist
$\{V_t^{(m)}\}_{t\in[-1,0]}$ and $\{u^{(m)}(\cdot,t)\}_{t\in[-1,0]}$ satisfying the assumptions of the present theorem with $\varepsilon=1/m$, and 
$x_m\in B_1^k$ such that 
$\tilde R^{-k}\|V_{-3/5}^{(m)}\|(\phi_{T,\tilde R,x_m}^2)>{\bf c}+\delta$. 
Again by the almost monotone property, we have 
\begin{equation}\label{ab1}
\inf_{t\in[-4/5,-3/5]}\tilde R^{-k}\|V_t^{(m)}\|(\phi_{T,\tilde R,x_m}^2)\geq {\bf c}+\delta/2
\end{equation}
for all large $m$. Since \begin{equation*}\int_{-4/5}^{-3/5}\int_{\rC(T,3/2)}|h(V^{(m)}_t,\cdot)|^2\,d\|V^{(m)}_t\|dt
\end{equation*}
is uniformly bounded by \eqref{e:energy_monotonicity} and (A2), using Fatou's lemma and (A1) we conclude that for almost all $t_0\in[-4/5,-3/5]$, 
there exists a subsequence $V_{t_0}^{(m_j)}\in {\bf IV}_k(\rC(T,2))$
such that the $L^2(\|V_{t_0}^{(m_j)}\|)$-norms of $\{h(V_{t_0}^{(m_j)})\}_j$ are bounded
uniformly in $\rC(T,3/2)$. Then, by Allard's compactness 
theorem of integral varifolds, a further subsequence 
converges to $V\in{\bf IV}_k(\rC(T,3/2))$, and due to \eqref{nonsup}, 
it is supported on $T$. Since the $L^2$-norm
of the generalized mean curvature is lower-semicontinuous
under varifold convergence, $V$ has $h(V,\cdot)\in L^2(\|V\|)$
in $\rC(T,3/2)$ and the multiplicity of $V$ on $T$ has to be
a constant function with integer value, and by \eqref{ab1}, the integer
has to be $\geq 2$. But this implies that
$\liminf_{j\rightarrow\infty}\|V_{t_0}^{(m_j)}\|
(\phi_T^2)\geq \|V\|(\phi_T^2)\geq 2{\bf c}$. 
Since $t_0\geq -4/5$ and by the almost monotone property, one can
obtain a contradiction to \eqref{udef1} for all large $m_j$.
This proves \eqref{ab3}. Similarly, we claim
that, given $\delta>0$, for small $\Cr{e4}>0$,
\begin{equation}\label{ab2}
    \tilde R^{-k}\|V_t\|(\phi_{T,\tilde R, x}^2)\geq 
    {\bf c}-\delta
\end{equation}
for all $(x,t)\in B_1^k\times [-3/5,-\iota/2]$. 
Again by the almost monotone property, we need to 
prove the claim at $t=-\iota/2$. The similar contradiction argument 
applied to the time interval $[-\iota/2,-\iota/4]$
in place of $[-4/5,-3/5]$ (with the same notation) shows that, for almost all $t_0\in[-\iota/2,-\iota/4]$,
there exists a subsequence such that $\lim_{j\rightarrow\infty}
\|V_{t_0}^{(m_j)}\|=0$ on $\rC(T,3/2)$. But then, with the 
clearing-out lemma (see \cite[Corollary 6.3]{Kasai-Tone}),
one can show that
$(\|V_t^{(m_j)}\|\times dt)(\rC(T,1)\times(-\iota/8,0))=0$ for all large $j$ (where $\iota$ needs to be smaller than 
a constant depending only on $k,n,p,q$ and $E_1$ for the clearing-out
lemma). 
This is a contradiction to \eqref{udef2}. This proves \eqref{ab2}.
Now we are ready to apply \cite[Theorem 8.7]{Kasai-Tone}:
we choose a small $\delta>0$ so that ${\bf c}-\delta>{\bf c}/2$
and ${\bf c}+\delta<3{\bf c}/2$ and 
let $\Cr{e4}$ be restricted
so that we have \eqref{ab3} and \eqref{ab2}.
Then for each $T^{-1}(B^k_{3\tilde R}(x))\times(t-\Cr{lam1}\tilde R^2,t+\Cr{lam1}\tilde R^2)$
with $(x,t)\in B_1^k\times[-1/2,-\iota]$, 
all the assumptions for \cite[Theorem 8.7]{Kasai-Tone} are
satisfied. Thus the support of $\|V_t\|$ can be represented
as the graph of a $C^{1,\zeta}$ function in $T^{-1}(B^k_{\sigma\tilde R}(x))
\times (t-\tilde R^2/4,t+\tilde R^2/4)$ with estimate
in terms of $\mu$ and $\|u\|$. Since $\rC(T,1)\times[-1/2,-\iota]$ can be 
covered by a finite number of such domains, the
support of the flow is represented as a $C^{1,\zeta}$ 
graph over $B^k_1\times[-1/2,-\iota]$ with estimates
in terms of $\mu$ and $\|u\|$. The resulting
constant $\Cr{c4}$ depends only on $E_1,\,\nu,\,p,\,q,\,\iota$.
This concludes the proof.
\end{proof}
The constants in the claim of Proposition \ref{blow-es} deteriorate as
$\iota$ approaches to $0$, and we will use it
with a fixed $\iota$ depending only on $E_1$, $\nu$ and 
$\zeta$
in Proposition \ref{p:las}. We next prove
the main decay estimate under the parabolic
dilation centered at the end-time, which 
will be iterated to obtain the desired $C^{1,\zeta}$
estimate.
\begin{proposition}\label{blow-est}
Corresponding to $E_1\in[1,\infty)$, $\nu\in(0,1)$, $p$ and $q$ there exist
$\Cl[eps]{e6}\in(0,1)$, $\theta\in(0,1/4)$ and $\Cl[con]{c8}\in(1,\infty)$
with the following property. For $W\in{\bf G}(n,k)$,
$0<R<\infty$ and $U=\rC(W,2R)$, suppose that $\{V_t\}_{t\in[-R^2,0]}$ and $\{u(\cdot,t)\}_{t\in[-R^2,0]}$
satisfy (A1)-(A4). 
Suppose 
\begin{equation}\label{bl1}
    T\in{\bf G}(n,k)\,\,\mbox{satisfies}\,\,\|T-W\|<\Cr{e6}\,,
\end{equation}
\begin{equation}\label{bl2}
    A\in{\bf A}(n,k)\,\,\mbox{is parallel to}\,\,T\,,
\end{equation}
\begin{equation}\label{bl3}
    \mu:=\Big(R^{-k-4}\int_{-R^2}^0\int_{\rC(W,2R)}
    {\rm dist} \,(x,A)^2\,d\|V_t\|dt\Big)^{1/2}<\Cr{e6}\,,
\end{equation}
\begin{equation}\label{bl4}
    \|u\|:=R^\zeta\|u\|_{L^{p,q}(\rC(W,2R)\times(-R^2,0))}<\infty\,,
\end{equation}
\begin{equation}\label{bl5}
    (\rC(W,\nu R)\times\{0\})\cap {\rm spt}\,(\|V_t\|\times dt)\neq \emptyset\,,
\end{equation}
\begin{equation}\label{bl6}
    R^{-k}\|V_{-4R^2/5}\|(\phi_{W,R}^2)\leq (2-\nu){\bf c}\,.
\end{equation}
Then there are $\tilde T\in{\bf G}(n,k)$ and $\tilde A\in 
{\bf A}(n,k)$ such that 
\begin{equation}\label{bl7}
    \tilde A\mbox{ is parallel to }\tilde T\,,
\end{equation}
\begin{equation}\label{bl8}
    \|T-\tilde T\|\leq \Cr{c8}\mu\,,
\end{equation}
\begin{equation}\label{bl9}
    \left((\theta R)^{-(k+4)}\int_{-(\theta R)^2}^0
    \int_{\rC(W,2\theta R)}
    {\rm dist}\,(x,\tilde A)^2\,d\|V_t\|dt\right)^{1/2}
\leq \theta^{\zeta}\max\{\mu,\Cr{c8}\|u\|\}\,.
\end{equation}
Moreover, if $\|u\|<\Cr{e6}$, we have
\begin{equation}\label{bl10}
    (\theta R)^{-k}\|V_{-4(\theta R)^2/5}\|(\phi^2_{W,\theta R})\leq (2-\nu){\bf c}\,.
\end{equation}
\end{proposition}
\begin{proof}
We may assume that $R=1$ after a parabolic change of variables.
The outline of proof is similar to \cite[Proposition 8.1]{Kasai-Tone}, with the crucial difference that we work with 
\eqref{bl5} and that the result is for a domain centered at the end-time point $(x,t)=(0,0)$. We give
a description on the different points on the proof for this result. The proof proceeds by contradiction.
We will fix $\theta\in(0,1/4)$ later depending only on $E_1$
and $\zeta$. If the claim were false, then, for each $m\in\mathbb N$ there exist $\{V_t^{(m)}\}_{t\in[-1,0]}$, $\{u^{(m)}(\cdot,t)\}_{t\in[-1,0]}$
satisfying (A1)-(A4) on $\rC(W^{(m)},2)\times[-1,0]$ for $W^{(m)}\in{\bf G}(n,k)$ such that, by assuming $T=\mathbb R^k\times\{0\}$ after 
suitable rotation,
\begin{equation}\label{til1}
    \|T-W^{(m)}\|\leq 1/m,
\end{equation}
\begin{equation}
    \mu^{(m)}:=\Big(\int_{-1}^0\int_{\rC(W^{(m)},2)} |T^{\perp}(x)|^2\,d\|V_t^{(m)}\|dt\Big)^{1/2}\leq 1/m,
\end{equation}
\eqref{bl5} and \eqref{bl6}, but for any $\tilde T\in{\bf G}(n,k)$ with
$\|T-\tilde T\|\leq m\mu^{(m)}$ and $\tilde A\in{\bf A}(n,k)$ 
which is parallel to $\tilde T$, we have
\begin{equation} \label{non1}
    \Big(\theta^{-(k+4)}\int_{-\theta^2}^0\int_{\rC(W^{(m)},2\theta)}
{\rm dist}\,(x,\tilde A)^2\,d\|V_t^{(m)}\|dt\Big)^{1/2}>\theta^\zeta\max\{\mu^{(m)},m\|u^{(m)}\|\}.
\end{equation}
By taking $\tilde A=\tilde T=T$ in \eqref{non1},
we obtain
\begin{equation*}
    \theta^\zeta\|u^{(m)}\|<\theta^{-(k+4)/2}m^{-1}
    \mu^{(m)},
\end{equation*}
which shows in particular that
\begin{equation}\label{non2}
    \lim_{m\rightarrow\infty} (\mu^{(m)})^{-1}
    \|u^{(m)}\|=0.
\end{equation}
By \eqref{non2}, \eqref{nonsup2} and \eqref{nonsup}, we have
\begin{equation}\label{non7}
    \limsup_{m\rightarrow\infty} \Big\{ \frac{|T^{\perp}(x)|}{\mu^{(m)}}\,:\, x\in {\rm spt}\|V_t^{(m)}\|\cap \rC(T,1)\Big\}\leq \sqrt{\Cr{c6}}
\end{equation}
for all $t\in[-4/5,0]$, where $\sqrt{\Cr{c6}}=c(n,k)$.
The same argument used to prove \eqref{ab3} combined with \eqref{bl6} shows
\begin{equation}\label{non11}
    \limsup_{m\rightarrow\infty}\|V_{-7/10}^{(m)}\|(\phi_T^2)
    \leq {\bf c}.
\end{equation}
Using \eqref{til1} and the similar argument leading 
to \eqref{ab1}, one can prove that
\begin{equation}\label{non6}
    \liminf_{m\rightarrow\infty} \|V_{-\theta^6/2}^{(m)}\|
    (\phi_T^2)\geq {\bf c}.
\end{equation}
Then, with \eqref{non7}-\eqref{non6}, for all
sufficiently large $m$, we may 
apply Theorem \ref{t:mass_growth_control}
with $\tau=\theta^6/2$.
Thus there exists a constant $\Cl[con]{c9}=\Cr{c9}(\theta,\nu,p,q,E_1)$ such that
\begin{equation}\label{non12}
   \limsup_{m\rightarrow\infty} \Big(\sup_{t\in[-3/5-\theta^6,-\theta^6]}
    (\mu^{(m)})^{-2}\big|\|V_t^{(m)}\|(\phi_T^2)-{\bf c}\big|\Big)\leq \Cr{c9}.
\end{equation}
We now apply Proposition \ref{lipap} with 
the time interval shifted from $[-3/5,0]$ to $[-3/5-\theta^6,-\theta^6]$. 
For all sufficiently large $m$, note that
\eqref{lipap2}-\eqref{lipap4} are all satisfied due to \eqref{non12}, \eqref{non7} and \eqref{non2}. The smallness condition of \eqref{lipap1} 
can be proved by (A4) and \eqref{non12} as it was done for \eqref{e:energy_monotonicity}.
Thus we have Lipschitz functions $f^{(m)}$ and
$F^{(m)}$ defined on $B^k_{1/3}\times[-1/2-\theta^6,-\theta^6]$
with quantities \eqref{lipap5} and \eqref{lipap6} 
defined in terms of $V^{(m)}$ and where $f^{(m)}$
and $F^{(m)}$ satisfy \eqref{lipap7}-\eqref{lipap9}.
Once we achieve this, arguing exactly 
as in \cite[p.45]{Kasai-Tone}, one can
prove that the right-hand side of \eqref{lipap9}
corresponding to $V^{(m)}$ can be bounded by 
$c(\mu^{(m)})^2$ with $c$ depending only on $\theta,\,\nu,\,E_1,\,
p,\,q$. We define the blowup sequence by 
\begin{equation}
    {\tilde f}^{(m)}:=f^{(m)}/\mu^{(m)}
\end{equation}
for all sufficiently large $m$ on $B^k_{1/3}\times[-1/2-\theta^6,-\theta^6]$. 
Writing $\Omega':=B^k_{1/3}\times(-1/2-\theta^6,-\theta^6]$,
the verbatim proof for 
\cite[Lemma 8.3,\,8.4]{Kasai-Tone}
gives the existence of a subsequence 
$\{\tilde f^{(m_j)}\}$ and $\tilde f\in C^{\infty}(\Omega')$ such that
\begin{equation}
    \lim_{j\rightarrow\infty}\|\tilde f^{(m_j)}-\tilde f\|_{L^2(\Omega')}=0
    \,\,\,\mbox{ and }\,\,\,
    \frac{\partial\tilde f}{\partial t}-\Delta 
    \tilde f=0\,\,\mbox{ on }\Omega'.
\end{equation}
At this point, it is important to note that
\eqref{non7} gives
\begin{equation}
    \|\tilde f\|_{L^\infty(\Omega')}\leq \sqrt{\Cr{c6}}\,,
\end{equation}
where $\Cr{c6}=c(n,k)$. We then define $T^{(m)}\in{\bf G}(n,k)$
as the graph of the map
\[
x \in \mathbb{R}^k \mapsto \mu^{(m)}\nabla \tilde f(0,-\theta^6) \cdot x \in \mathbb{R}^{n-k}\,,
\]
which is the tangent space
to the graph $\{(x,\mu^{(m)}\tilde f(x,-\theta^6)) \, \colon \, x \in B^k_{1/3}\}$ at $x=0$, 
and also define the affine plane $A^{(m)}\in{\bf A}(n,k)$
by $A^{(m)}=T^{(m)}+(0,\mu^{(m)}\tilde f(0,-\theta^6))$.
By the standard estimates for parabolic PDE, all the partial
derivatives of $\tilde f$ 
on $B^k_{2\theta}\times[-\theta^2,-\theta^6]$
are bounded in terms of constant multiple of $\sqrt{\Cr{c6}}$. 
In particular, there exists a constant $\Cl[con]{c10}=c(n,k)$ such that
\begin{equation}
\int_{B_{2\theta}\times[-\theta^2,-\theta^6]}
|\tilde f(x,t)-\tilde f(0,-\theta^6)-\nabla\tilde f
(0,-\theta^6)\cdot x|^2\,d\mathcal H^k\leq \Cr{c10}\theta^{k+6}.
\end{equation}
Following the verbatim proof in \cite{Kasai-Tone},
this leads to 
\begin{equation}
\begin{split}
   &\limsup_{m\rightarrow\infty} \|T-T^{(m)}\|\leq \Cr{c10}, \\
   & \limsup_{m\rightarrow\infty} (\mu^{(m)})^{-2}
   \int_{\rC(T,2\theta)\times(-\theta^2,-\theta^6)}
   {\rm dist}\,(x,A^{(m)})^2\,d\|V_t^{(m)}\|dt\leq \Cr{c10}
   \theta^{k+6}.
\end{split}
\end{equation}
Thus, for all large $m$, we have
\begin{equation}\label{non13}
   \theta^{-(k+4)}
   \int_{\rC(T,2\theta)\times(-\theta^2,-\theta^6)}
   {\rm dist}\,(x,A^{(m)})^2\,d\|V_t^{(m)}\|dt\leq \Cr{c10}
   \theta^{2} (\mu^{(m)})^{2}.
\end{equation}
On the integral over the time interval $(-\theta^6,0)$, 
since ${\rm dist}\,(x,A^{(m)})\leq c(\Cr{c10})\mu^{(m)}$
on the support of $\|V_t^{(m)}\|$, combined with (A2), we have
\begin{equation}\label{non14}
    \theta^{-(k+4)}\int_{\rC(T,2\theta)\times(-\theta^6,0)}
   {\rm dist}\,(x,A^{(m)})^2\,d\|V_t^{(m)}\|dt\leq \Cl[con]{c11}
   \theta^{2} (\mu^{(m)})^{2}
\end{equation}
where $\Cr{c11}$ depends only on $\Cr{c10}$ and $E_1$. Then \eqref{non13}
and \eqref{non14} show
\begin{equation}\label{non15}
   \theta^{-(k+4)}
   \int_{\rC(T,2\theta)\times(-\theta^2,0)}
   {\rm dist}\,(x,A^{(m)})^2\,d\|V_t^{(m)}\|dt\leq (
   \Cr{c10}+
   \Cr{c11})
   \theta^{2} (\mu^{(m)})^{2}. 
\end{equation}
Now, choosing $\theta$ small enough depending only on 
$n,\,k,\,E_1,\zeta$, we may assume that $(\Cr{c10}+
   \Cr{c11})\theta^2
<\theta^{2\zeta}/2$. Since $T$ can be 
replaced by $W^{(m)}$ for the limit (see \cite{Kasai-Tone}) in \eqref{non15}, we have a contradiction to \eqref{non1}. This completes the proof of claims \eqref{bl7}-\eqref{bl9}. For
\eqref{bl10}, since $\theta$ is fixed, we may argue as for the proof of \eqref{ab3} and restrict
$\Cr{e6}$ to make sure that \eqref{bl10} holds. 
This completes the proof. 
\end{proof}
It is possible to apply Proposition \ref{blow-est} iteratively; in combination with Proposition \ref{blow-es}, we have then the following.
\begin{proposition} \label{p:las}
Corresponding to $E_1\in[1,\infty)$, $\nu\in(0,1)$,
$p$ and $q$, there exist $\Cl[eps]{e7}\in(0,1)$ and $\Cl[con]{c-fin} \in (1,\infty)$  with the following property. For
$T\in{\bf G}(n,k)$, $R\in(0,\infty)$ and $U=\rC(T,2R)$,
suppose that $\{V_t\}_{t\in[-R^2,0]}$ and 
$\{u(\cdot,t)\}_{t\in[-R^2,0]}$ satisfy (A1)-(A4). 
Suppose
\begin{equation}\label{las1}
    \mu:=\Big(R^{-k-4}\int_{-R^2}^0
    \int_{\rC(T,2R)}|T^\perp(x)|^2\,d\|V_t\|dt\Big)^{1/2}<\Cr{e7}\,,
    \end{equation}
\begin{equation}\label{las2}
    \|u\|:=R^\zeta \|u\|_{L^{p,q}(\rC(T,2R)\times(-R^2,0))}<\Cr{e7}\,,
\end{equation}
\begin{equation}\label{las3}
    (T^{-1}(0)\times\{0\})\cap{\rm spt}(\|V_t\|\times dt)\neq \emptyset\,,
\end{equation}
\begin{equation}\label{las4}
    R^{-k}\|V_{-4R^2/5}\|(\phi_{T,R}^2)\leq (2-\nu){\bf c}\,.
\end{equation}
Identifying $T$ as $\mathbb R^k\cong\mathbb R^k\times\{0\}\subset\mathbb R^n$, 
let $\tilde D:=\{(x,t)\in \mathbb R^k\times[-R^2/2,0)\,:\, |x|^2<
|t|\}$. Then there exist $f\,:\,\tilde D
\rightarrow T^{\perp}$ and $F\,:\,\tilde D\rightarrow\mathbb R^n$ such 
that $F(x,t)=(x,f(x,t))$ for $(x,t)\in\tilde D$ and 
\begin{itemize}
    \item[(1)] ${\rm spt}\|V_t\|\cap \rC(T,\sqrt{|t|})={\rm Image}\,
    F(\cdot,t)\,\,\mbox{for all }\, t\in[-R^2/2,0)$,
    \item[(2)]
    $R^{-1}\|f\|_0+\|\nabla f\|_0+R^\zeta[ f]_{1+\zeta}\leq \Cr{c4}\Cr{c-fin}\max\{\mu,\Cr{c8}\|u\|\}$.
\end{itemize}
\end{proposition}
\begin{proof} We may set $R=1$ without loss of generality. 
With $E_1$, $\nu$, $p$ and $q$ given, we use Proposition \ref{blow-est} to obtain $\Cr{e6}$, $\theta$ and $\Cr{c8}$.
Setting $\iota=\theta^2/2$, we use Proposition \ref{blow-es}
to obtain $\Cr{e4}$ and $\Cr{c4}$. We choose $\Cr{e7}$ so that
\begin{equation}\label{eps6}
    \Cr{e7} \leq \min\{\Cr{e4}, \Cr{e6}\}\,,
\end{equation}
\begin{equation}\label{eps6-1}
    \Cr{c8}\Cr{e7}<\Cr{e6}\,,
\end{equation}
\begin{equation}\label{eps6-2}
    (\Cr{c8})^2(1-\theta^\zeta)^{-1}\Cr{e7}<\Cr{e6}\,.
\end{equation}
We first use Proposition \ref{blow-est} with $W=A=T$, and 
note that \eqref{bl1}-\eqref{bl6} are satisfied due to \eqref{las1}-\eqref{las4} and \eqref{eps6}. Thus there 
exist $T_1\in{\bf G}(n,k)$ and $A_1\in{\bf A}(n,k)$ 
such that \eqref{bl7}-\eqref{bl9} are satisfied with 
$R=1$, $W=T$, $\tilde A=A_1$ and $\tilde T=T_1$. 
Similarly, we may use Proposition
\ref{blow-es} since \eqref{udef1}-\eqref{udef4}
are satisfied with $R=1$ and $\Cr{e4}$,
so that we have $f_1$ and $F_1$ defined on $B_1^k\times[-1/2,-\theta^2/2]$ satisfying \eqref{f5} and \eqref{f6}. We next claim that Proposition \ref{blow-est} can
be inductively used for $R=\theta^j$, $j\in\mathbb N$, where
we obtain $T_j\in{\bf G}(n,k)$ and $A_j\in{\bf A}(n,k)$ satisfying
\begin{equation}\label{las5}
    \|T_j-T_{j-1}\|\leq \Cr{c8} \theta^{(j-1)\zeta}\max\{\mu,\Cr{c8}\|u\|\},
\end{equation}
where $T_0:=T$, and writing $\mu_j$ as $\mu$ in \eqref{bl3} corresponding to $A_j$ and $R=\theta^j$,
\begin{equation}\label{las6}
    \mu_j\leq \theta^{j\zeta}\max\{\mu,\Cr{c8}\|u\|\}.
\end{equation}
The case $j=1$ follows from Proposition \ref{blow-est}. 
Assume that it is true until $j\geq 1$. 
Then we check that \eqref{bl1}-\eqref{bl6} are true for
$W=T$, $T=T_j$, $A=A_j$ and $R=\theta^j$. We have
\begin{equation}\label{las8}
    \|T_j-T\|\leq \sum_{l=1}^j\|T_{l}-T_{l-1}\|\leq 
    \Cr{c8}\sum_{l=1}^j \theta^{(l-1)\zeta}\max\{\mu,\Cr{c8}\|u\|\}
    \leq (\Cr{c8})^2(1-\theta^\zeta)^{-1}\Cr{e7}<
    \Cr{e6}\end{equation}
where we used \eqref{las5}, \eqref{las1}, \eqref{las2} and
\eqref{eps6-2}. Thus \eqref{bl1} is satisfied. Since $A_j$ 
and $T_j$ are parallel, \eqref{bl2} is fine. By 
\eqref{eps6}, \eqref{eps6-1} and \eqref{las6}, we have
$\mu_j<\Cr{e6}$, so that \eqref{bl3} is satisfied.
The condition \eqref{bl5} follows from \eqref{las3}, and
\eqref{eps6}, \eqref{las2} and \eqref{bl10} give \eqref{bl6} 
for $j$. Thus, we may apply Proposition \ref{blow-est} with
$R=\theta^j$, and obtain $T_{j+1}$ and $A_{j+1}$ which are parallel
and 
\begin{equation}
    \|T_{j+1}-T_j\|\leq \Cr{c8}\mu_j\leq \Cr{c8}\theta^{j\zeta}\max\{\mu,\Cr{c8}\|u\|\}, 
\end{equation}
where we used \eqref{las6}, and 
\begin{equation}
    \mu_{j+1}\leq \theta^\zeta\max\{\mu_j,\theta^{j\zeta}\Cr{c8}\|u\|\}\leq \theta^{(j+1)\zeta}\max\{\mu,\Cr{c8}\|u\|\}
\end{equation}
by \eqref{bl9} and \eqref{las6}. This closes the inductive step 
and proves \eqref{las5} and \eqref{las6} for all $j$. 
We next prove that we can apply Proposition \ref{blow-es}
on each domain $\rC(T_j,2\theta^j)\times[-\theta^{2j},-\theta^{2(j+1)}/2]$ for all
$j\geq 1$. Note that for each $j\geq 0$, by the same
argument leading up to \eqref{nonsup}, we have
\begin{equation}\label{las7}
\begin{split}
    {\rm spt}\|V_t\|\cap \rC(T,3\theta^j/2)\}&\subset\{x\,:\,
    \theta^{-2j}{\rm dist}(x,A_j)^2\leq \Cr{c6} \mu_j^2+3 \Cr{c5}\theta^{2j\zeta}\|u\|^2
    E_1^{1-2/p}\} \\
    &\subset\{x\,:\,{\rm dist}(x,A_j)\leq \theta^{j(1+\zeta)}\Cl[con]{c-1}
    \Cr{e7}\} \,\,\, (\Cr{c-1}=\Cr{c-1}(n,k,E_1))
    \end{split}
\end{equation}
for all $t\in[-4\theta^{2j}/5,0)$.
% By \eqref{las3} and \eqref{las7}, there exists a 
% unique point $a\in(T^{-1}(0)\times\{0\})\cap {\rm spt}(\|V_t\|
% \times dt)$ such that
% \begin{equation}\label{las9}
%     {\rm dist}(a,A_j\cap T^{-1}(0))\leq \theta^{j(1+\zeta)}c\varepsilon_6.
% \end{equation}
To apply Proposition \ref{blow-es}, we need to have $T$ there 
replaced by $A_j$, so we need to tilt the plane whose tilt is 
estimated by \eqref{las8}. For this reason, we may actually 
need to use a slightly smaller cylinder than $C(T_j,2\theta^j)$ 
so that the smallness of corresponding $\mu_j$ (with respect to the distance function to $A_j$)
can be assured from \eqref{las6}. Inductively, we know that
the support of $\|V_t\|$ in 
$\rC(T_{j-1},\theta^{j-1})\times[-\theta^{2(j-1)}/2,-\theta^{2j}/2]$
is a $C^{1,\zeta}$ graph, so that the condition \eqref{udef1} is satisfied. Condition \eqref{udef2} follows from \eqref{las3},
and \eqref{udef3} follows from \eqref{las6},
\eqref{eps6} and \eqref{eps6-1}. 
Thus we may apply Proposition \ref{blow-es} 
and obtain
a graph representation $\tilde f_j$ over $A_j$ with the 
$C^{1,\zeta}$ estimate of the form $\Cr{c4}\theta^{j\zeta}
\max\{\mu,\Cr{c8}\|u\|\}$. Note that, by the implicit function 
theorem, one can equally represent the same set as a graph 
$f_j$ over $T$. The norm $\|\nabla f_j\|_0$ over
$B_{\theta^j}^k\times[-\theta^{2j}/2,-\theta^{2(j+1)}/2]$ can be 
different by a constant multiple of $\|T_j-T\|$ which is 
bounded as in \eqref{las8}. The H\"{o}lder semi-norm $[f]_{1+\zeta}$ has two terms, 
$[\nabla f]_\zeta$ and the $(1+\zeta)/2$-H\"{o}lder semi-norm in time. The first is 
seen as the variation of the tangent space and one can see that it is bounded by a multiple of constant (which is close to $1$) under the small rotation. The estimate for the latter is obtained by 
applying \cite[Proposition 6.4]{Kasai-Tone} with the gradient H\"{o}lder norm, and 
the small rotation affects little. 
Hence we can obtain the desired 
$C^{1,\zeta}$ estimate for $f_j$ representing 
${\rm spt}\|V_t\|$ over the domain
$B_{\theta^j}^k\times[-\theta^{2j}/2,-\theta^{2(j+1)}/2]$, by $2\Cr{c4}\theta^{j\zeta}\max\{\mu,\Cr{c8}\|u\|\}$.
We next observe that
\begin{equation}
   \tilde D=\{(x,t)\in\mathbb R^k\times[-1/2,0)\,:\,|x|^2<|t|\}\subset
   \cup_{j=0}^\infty B_{\theta^j}^k\times[-\theta^{2j}/2,
   -\theta^{2(j+1)}/2]\,,
\end{equation}
so that we have a representation of $\spt\|V_t\|$ as the graph of a single function $f$ over $\tilde D$. The estimate $\|f\|_0+\|\nabla f\|_0 \leq 2 \Cr{c4} \max\{\mu, \Cr{c8} \|u\|\}$ is immediate. For the H\"older semi-norm $[f]_{1+\zeta}$, we proceed as follows. Let $(y_1,s_1)$, $(y_2,s_2)$ be points in $\tilde D$ with $(y_1,s_1) \neq (y_2,s_2)$, assume without loss of generality that $s_1 \leq s_2$, and let $h,l \geq 0$ be such that $(y_1,s_1) \in B^k_{\theta^h} \times [ -\theta^{2h}/2,-\theta^{2(h+1)}/2]$ and $(y_2,s_2) \in B^k_{\theta^{h+l}} \times [ -\theta^{2(h+l)}/2,-\theta^{2(h+l+1)}/2]$. By the triangle inequality, we estimate
\[
\begin{split}
\left| \nabla f(y_1,s_1) - \nabla f(y_2,s_2) \right| &\leq 2 \Cr{c4} \max\{\mu,\Cr{c8}\|u\|\} \left( |y_1-y_2|^\zeta + \frac12 \sum_{j=h}^{h+l} (\theta^{2j})^{\zeta/2}  \right)\\
&\leq  \Cr{c4} \Cr{c-fin} \max\{\mu,\Cr{c8}\|u\|\} \theta^{h\zeta} \\
&\leq  \Cr{c4} \Cr{c-fin} \max\{\mu,\Cr{c8}\|u\|\} |s_1-s_2|^{\zeta/2}\,,
\end{split}
\]
where $\Cr{c-fin}=\Cr{c-fin}(k,p,q)$. The estimate for the second summand in $[f]_{1+\zeta}$ is analogous. The proof is now complete.
\end{proof}

\section{Proof of the main results}
\label{PMR}
We are now ready to prove Theorem \ref{main2} and Theorem \ref{main3}.

\begin{proof}[Proof of Theorem \ref{main2}]

By scaling, we may assume $R=1$. Given $\nu \in \left(0,1\right)$, $E_1 \in \left[1, \infty \right)$, $p$ and $q$, let $\Cr{e7}$, $\Cr{c4}$, $\Cr{c-fin}$ and $\Cr{c8}$ be as in Proposition \ref{p:las}. Let now $\Cr{e1} \in \left(0,1\right)$ and $c_1 \in \left( 1, \infty \right)$ be such that the following conditions are satisfied:
\begin{align}
    \Cr{e1} \leq \frac{\Cr{e7}}{2^{k+4}}\,, \qquad c_1 \geq 4\,\max\{2^{k+4}\Cr{c4}\Cr{c-fin},\Cr{c4}\Cr{c-fin}\Cr{c8}\}\,. \label{e:eps_cond}
\end{align}

For $T \in \mathbf{G}(n,k)$, and $U = \rC(T,2)$, suppose that $\{V_t\}_{t \in \left[-1,0\right]}$ and $\{u(\cdot,t)\}_{t \in \left[-1,0\right]}$ satisfy (A1)-(A4) as well as \eqref{udef1}-\eqref{udef4}. We identify, as usual, $T$ with $\R^k \cong \R^k \times \{0\} \subset \R^n$, and we claim the following: for every $j \geq 1$, setting
\begin{align} \label{e:parameters}
& \sigma_j := \sum_{i=1}^j \frac{1}{i}\,, \qquad \tau_1 := \frac{1}{2} \,, \qquad \tau_{j+1} := \frac{1}{4\sigma_j} \,, \\ \label{e:increasing_parabolas}
  &  D_j := \left\lbrace (x,t) \in \R^k \times \left[ -\tau_j,0\right) \, \colon \, |x|^2 < \sigma_j |t| \right\rbrace\,, 
\end{align}
there exist $f_j \colon D_j \to T^\perp$ and $F_j \colon D_j \to \R^n$ such that $F_j(x,t)=(x,f_j(x,t))$, and
\begin{enumerate} 
   \item $\spt \|V_t\| \cap \rC(T,\sqrt{\sigma_j|t|}) = \mathrm{Image}\, F_j(\cdot,t) \mbox{ for all } t \in \left[-\tau_j,0\right)$,
   \item $\|f_j\|_0 + \|\nabla f_j\|_0 \leq c_1 \max\{\mu,\|u\|_{p,q}\}$. 
\end{enumerate}
Assume the claim for the moment. It is then an immediate consequence of \eqref{e:parameters} that
\[
\sqrt{\sigma_j |t|} \geq 1/2 \mbox{ for all $t \in \left[-\tau_j, -\tau_{j+1} \right)$}\,,
\]
which implies that
\begin{equation} \label{e:good-cylinder}
    B^k_{\frac12} \times \left[-\tau_j, -\tau_{j+1} \right) \subset D_j\,.
\end{equation}
Since $\lim_{j \to \infty} \tau_j = 0$, \eqref{e:good-cylinder} and (1)-(2) imply that one can define a function $f \colon B^k_{\frac12} \times \left[-\frac14,0\right) \to T^\perp$ such that, setting $F(x,t)=(x,f(x,t))$ for $(x,t) \in B^k_{\frac12} \times \left[-\frac14,0\right)$ one has 
\begin{align*}
    & \spt\|V_t\| \cap \rC(T,1/2) = \mathrm{image}\,F(\cdot, t) \mbox{ for all $t\in \left[-1/4,0\right)$}\,, \\
    & \|f\|_0 + \| \nabla f \|_0 \leq c_1 \max\{\mu,\|u\|_{p,q}\}\,.
\end{align*}
that is \eqref{udef5-12} and part of the estimate in \eqref{udef5}. In what follows, we will first prove the claim; then, we will show that the resulting function $f$ also satisfies $[f]_{1+\zeta} \leq c_1 \max\{\mu,\|u\|_{p,q}\}$.

The proof of the claim is by induction on $j\geq 1$. The induction base, $j=1$, is Proposition \ref{p:las}. We then assume that the claim is true for $j$, and prove it for $j+1$. Fix any point $(x_0,t_0) \in \partial D_j$, and translate in space-time so to consider the flow $\{\tilde V_s\}_{s\in\left[-1-t_0,0\right]}$, with $\tilde V_s := (\tau_{x_0})_\sharp V_{s+t_0}$ where $\tau_{x_0}(y) := y-x_0$. Set $\tilde R^2 = \tilde R_{t_0}^2 := \frac{1}{4}+\frac{t_0}{4}$, and notice that $\rC(T,x_0,2 \tilde R) \subset \rC(T,0,2)$. In particular, $\{\tilde V_s\}$ satisfies (A1)-(A4) in $U=\rC(T,2\tilde R)$ corresponding to the forcing term $\tilde u(y,s) = \tilde u_{(x_0,t_0)}(y,s) := u(y+x_0,s+t_0)$. We next claim that \eqref{las1}-\eqref{las4} are satisfied. We clearly have
\[
\mu_{(x_0,t_0)}^2 := \tilde R^{-k-4} \int_{-\tilde R^2}^0 \int_{\rC(T,2\tilde R)} |T^\perp(y)|^2\,d\|\tilde V_s\|(y)\,ds \leq \tilde R^{-k-4} \mu^2 \leq 4^{k+4}\mu^2 < \Cr{e7}^2
\]
by \eqref{udef1} and \eqref{e:eps_cond}. Moreover, $(T^{-1}(0)\times\{0\}) \cap \spt(\|\tilde V_s\|\times ds) = (T^{-1}(x_0)\times\{t_0\}) \cap \spt(\| V_t\|\times dt) \neq \emptyset$, because for any sequence $(x_h,t_0) \in D_j$ such that $x_h \to x_0$ we have $(x_h,f_j(x_h,t_0)) \in T^{-1}(x_h) \cap \spt\|V_{t_0\|}$ by (1), and thus $(T^{-1}(x_0) \times \{t_0\}) \cap \spt(\|V_t\|\times dt)$ contains all subsequential limits of $(x_h, f_j (x_h,t_0),t_0)$. We also readily estimate 
\[\|\tilde u\|_{L^{p,q}(\rC(T,2\tilde R) \times (-\tilde R^2,0))} \leq \|u\|_{L^{p,q}(\rC(T,2) \times (-1,0))}\,,
\]
so that \eqref{udef4} implies \eqref{las2}. Finally, we have
\[
\tilde R^{-k} \|\tilde V_{-4\tilde R^2/5}\|(\phi_{T,\tilde R}^2) = \tilde R^{-k} \|V_{-1/5+4t_0/5}\|(\phi^2_{T,\tilde R,x_0}) \leq \bc+\delta\,,
\]
using the same argument leading to \eqref{ab3}. We can then apply Proposition \ref{p:las} and conclude after translating back the origin to $(x_0,t_0)$ that, setting
\[
\tilde D^{(x_0,t_0)} := \left\lbrace (x,t) \in \R^k \times \left[ t_0 - \frac{\tilde R_{t_0}^2}{2}, t_0 \right) \, \colon \, |x-x_0|^2 < |t-t_0| \right\rbrace\,,
\]
there exist functions $f^{(x_0,t_0)} \colon \tilde D^{(x_0,t_0)} \to T^\perp$ and $F^{(x_0,t_0)} \colon \tilde D^{(x_0,t_0)} \to \R^n$ such that $F^{(x_0,t_0)}(x,t)\\=(x,f^{(x_0,t_0)}(x,t))$ for all $(x,t) \in \tilde D^{(x_0,t_0)}$ and
\begin{itemize}
    \item[$(1)_\star$] $\spt\|V_t\| \cap \rC(T,x_0,\sqrt{|t-t_0|}) = \mathrm{Image}\,F^{(x_0,t_0)}(\cdot,t)$ for all $t \in \left[ t_0 - \frac{\tilde R_{t_0}^2}{2}, t_0 \right)$,
    \item[$(2)_\star$] $\tilde R_{t_0}^{-1} \|f^{(x_0,t_0)}\|_0 + \|\nabla f^{(x_0,t_0)}\|_0 + \tilde R_{t_0}^\zeta [f^{(x_0,t_0)}]_{1+\zeta} \leq \Cr{c4}\Cr{c-fin} \max\{\mu_{(x_0,t_0)},c_8 \|\tilde u_{(x_0,t_0)}\|\}$\,.
\end{itemize}
In particular, there is a well posed extension of the functions $f_j$ and $F_{j}$ to the region
\[
D_j \cup \bigcup_{(x_0,t_0) \in \partial D_j} \tilde D^{(x_0,t_0)}\,.
\]
We let $f_{j+1}$ and $F_{j+1}$ denote such extensions, and we proceed with the proof that conditions (1)-(2) hold true with $j+1$ in place of $j$. To this aim, it is sufficient to show the following: for $t \in \left[-\tau_{j+1},0\right)$ and $\sigma_j |t| \leq |x|^2 < \sigma_{j+1} |t|$, there exists $(x_0,t_0) \in \partial D_j$ such that $(x,t) \in \tilde D^{(x_0,t_0)}$. Once this is established, indeed, one immediately gains that 
\begin{equation} \label{e:opening_parabola}
    D_{j+1} \subset D_j \cup \bigcup_{(x_0,t_0) \in \partial D_j} \tilde D^{(x_0,t_0)}\,,
\end{equation}
see Figure \ref{fig:env_par}, and (1) at step $j+1$ follows immediately from (1) at step $j$ and $(1)_\star$, while (2) at step $j+1$ follows from (2) at step $j$ and $(2)_\star$ thanks to \eqref{e:eps_cond} 

\begin{figure}
\begin{tikzpicture}
\draw[-] (-4,0) -- (4,0) node[right] {$T$};
\draw[->] (0,-4) -- (0,0.7) node[above] {$t$};
\draw[scale=6, thick, domain=-0.707:0.707, smooth, variable = \x] plot ({\x},{-\x*\x});
\draw[thick] (-4.242,-3) -- (4.242,-3);
\node[] at (3,-2.4) {$D_1$};
\draw[scale=6,domain=-0.39:0.79, smooth, variable = \x] plot ({\x},{-0.04-(\x-0.2)*(\x-0.2)});
\filldraw[black] (1.2,-0.24) circle(1pt);
\draw[scale=6,domain=-0.79:0.39, smooth, variable = \x] plot ({\x},{-0.04-(\x+0.2)*(\x+0.2)});
\filldraw[black] (-1.2,-0.24) circle(1pt);
 \draw[scale=6,domain=-0.25:0.45, smooth, variable = \x] plot ({\x},{-0.01-(\x-0.1)*(\x-0.1)});
 \filldraw[black] (0.6,-0.06) circle(1pt);
\draw[scale=6,domain=-0.45:0.25, smooth, variable = \x] plot ({\x},{-0.01-(\x+0.1)*(\x+0.1)});
 \filldraw[black] (-0.6,-0.06) circle(1pt);
\draw[scale=6,very thick, domain=-0.5:0.5, smooth, variable = \x] plot ({\x},{-0.67*\x*\x});
\draw[very thick] (-3,-1) -- (3,-1);
\node[] at (2.5,-0.85) {\tiny{$D_2$}};
\end{tikzpicture}
\caption{An illustration of the first two parabolic regions $D_j$: the region $D_2$ is a subset of the union of $D_1$ with suitable parabolic regions $\tilde D^{(x_0,t_0)}$ having vertices at points $(x_0,t_0) \in \partial D_1$ (black dots in the graph). The region $D_3$ will be a subset of the union of $D_2$ with parabolic regions $\tilde D^{(x_0,t_0)}$ having vertices at points $(x_0,t_0) \in \partial D_2$. As $j$ grows, the opening of the regions $D_j$ increases, as it is defined by the parameter $\sigma_j \uparrow \infty$. The union of the regions $D_j$ contains the cylinder $B_{1/2}^k \times \left[-1/4,0\right)$, over which we can conclude graphical parametrization and corresponding estimates for the flow.} \label{fig:env_par}
\end{figure}
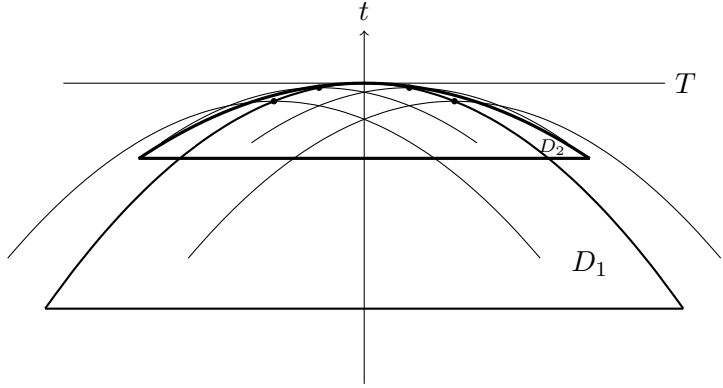

To prove the above claim, let then $(x,t) \in \R^{k} \times \left[ - \tau_{j+1},0\right)$ be such that $\sigma_j |t| \leq |x|^2 < \sigma_{j+1}|t|$, and set
\begin{equation} \label{t0x0}
    t_0 := \frac{t}{\alpha}\,, \qquad x_0 := \sqrt{\frac{\sigma_j |t|}{\alpha}} \frac{x}{|x|}
\end{equation}
for some number $\alpha=\alpha_j > 1$ to be determined. Notice that $(x_0,t_0) \in \partial D_j$ by construction. We then only need to prove that there exists $\alpha > 1$ such that $(x,t) \in \tilde D^{(x_0,t_0)}$. On the other hand, by the definitions of $t_0$ and $x_0$ it holds that
\begin{align*}
& |x-x_0| = |x| - \sqrt{\frac{\sigma_j |t|}{\alpha}} < \sqrt{|t|} \left( \sqrt{\sigma_{j+1}} - \sqrt{\frac{\sigma_j}{\alpha}} \right) \\ 
& \sqrt{|t-t_0|} = \sqrt{|t|} \sqrt{1-\frac{1}{\alpha}}\,, 
\end{align*}
so that, recalling the definition of $\sigma_j$, $(x_0,t_0) \in \tilde D^{(x_0,t_0)}$ provided $\alpha > 1$ is chosen so that
\begin{equation} \label{e:key-ineq}
    \sqrt{\sigma_j + \frac{1}{j+1}} - \sqrt{\frac{\sigma_j}{\alpha}} \leq \sqrt{1-\frac{1}{\alpha}}\,.
\end{equation}
We now show that \eqref{e:key-ineq} has a solution $\alpha=\alpha_j>1$ for every $j$. Direct calculation shows that $\alpha=2$ is a solution to \eqref{e:key-ineq} when $j=1$ and $j=2$. On the other hand, it holds
\[
 \sqrt{\sigma_j + \frac{1}{j+1}} - \sqrt{\frac{\sigma_j}{\alpha}} = \frac{\sigma_j \left(1-\frac{1}{\alpha}\right) + \frac{1}{j+1}}{\sqrt{\sigma_j + \frac{1}{j+1}} + \sqrt{\frac{\sigma_j}{\alpha}}} \leq \frac{\sigma_j \left(1-\frac{1}{\alpha}\right) + \frac{1}{j+1}}{\sqrt{\sigma_j}}\,,
\]
so that solutions to
\begin{equation} \label{e:key-ineq2}
    \sigma_j \left(1-\frac{1}{\alpha}\right) + \frac{1}{j+1} \leq \sqrt{\sigma_j \left(1-\frac{1}{\alpha}\right)}
\end{equation}
also solve \eqref{e:key-ineq}. Changing variable 
\[
\xi := \sqrt{\sigma_j \left(1-\frac{1}{\alpha}\right)}\,,
\]
\eqref{e:key-ineq2} reduces to 
\[
\xi^2 - \xi + \frac{1}{j+1} \leq 0\,,
\]
which admits $\xi = \frac12$ as a solution for every $j \geq 3$. Going back to the original variables, we have that the number $\alpha=\alpha_j > 1$ such that $\frac{1}{\alpha} = 1-\frac{1}{4\sigma_j}$ is a solution to \eqref{e:key-ineq} for $j \geq 3$. This concludes the proof of \eqref{e:opening_parabola}.

We are only left with the proof of the estimate on the H\"older semi-norm $[f]_{1+\zeta}$. Given that $\spt\|V_t\| \cap \rC(T,1/2)$ is the graph of a function defined on $B^k_{1/2}$ for all $t \in [-1/4,0)$, we know now that for every $(x_0,t_0) \in B^k_{1/2} \times [-1/4,0)$
the flow $\{\tilde V_s\}_{s \in [-1-t_0,0]}$
 with $\tilde V_s = (\tau_{x_0})_\sharp V_{s+t_0}$ as above satisfies the assumptions of Proposition \ref{p:las} with, say $R=3/4$. In particular, we have $C^{1,\zeta}$ estimates for $f$ with $\Cr{c4} \Cr{c-fin} \max\{\mu,\Cr{c8}\|u\|_{p,q}\}$ in the parabolic region $\tilde D^{(x_0,t_0)}=\{(x,t) \in \R^k \times [t_0 - 1/4,t_0) \, \colon \, |x-x_0|^2 < |t-t_0|\}$. To prove the desired H\"older estimate, let now $(y_1,s_1)$ and $(y_2,s_2)$ be points in $B^k_{1/2} \times [-1/4,0)$ with $(y_1,s_1) \neq (y_2,s_2)$, and assume without loss of generality that $s_1 \leq s_2$. Consider the parabolic region $\tilde D^{(y_2,s_2)}$ with vertex at $(y_2,s_2)$. If $|y_1-y_2|^2 < |s_1-s_2|$, then $(y_1,s_1) \in \tilde D^{(y_2,s_2)}$, and the estimate is a consequence of Proposition \ref{p:las} and \eqref{e:eps_cond}. Otherwise, if $|y_1-y_2|^2 \geq |s_1-s_2| = s_2-s_1$, we use the triangle inequality to estimate
\[
\begin{split}
    |\nabla f (y_1,s_1) - \nabla f (y_2,s_2)| &\leq  |\nabla f (y_2,s_2) - \nabla f (y_2,s_2-|y_1-y_2|^2)| \\
   & \qquad + |\nabla f (y_2,s_2-|y_1-y_2|^2) - \nabla f (y_1,s_2-|y_1-y_2|^2)| \\
   & \qquad + |\nabla f (y_1,s_2-|y_1-y_2|^2) - \nabla f(y_1,s_1)| \\
   &\leq \Cr{c4}\Cr{c-fin} \max\{\mu,\Cr{c8}\|u\|_{p,q}\} \left( |y_1-y_2|^\zeta + |y_1-y_2|^\zeta + 2^{\zeta/2} |y_1-y_2|^\zeta \right)\,,
\end{split}
\]
which yields the estimate for $[\nabla f]_\zeta$ thanks to \eqref{e:eps_cond}. The estimate for the second summand in $[f]_{1+\zeta}$ is analogous, and we omit it. The proof is complete.
\end{proof}

\begin{proof}[Proof of Theorem \ref{main3}]
Here we briefly record the outline of 
the $C^{2,\alpha}$ regularity of \cite{Ton-2}
and point out the key estimates. The idea
is to look at a graphical distance function from the
solution of the heat equation $g$, denoted by
$Q_g$ (\cite[Definition 4.1]{Ton-2}) and one
shows a decay estimate of the $L^2$-norm of 
$Q_g$ by the blowup argument. The key identity
is Lemma 4.2, which shows certain ``sub-caloric'' 
property of $Q_g$, and the resulting $L^\infty$
estimate Proposition 4.3, both of \cite{Ton-2}.
Note that the latter is an estimate up to the 
end-time. Since this is the basis of the blowup
argument, if we have already $C^{1,\zeta}$ graph 
representation up to the end-time, all the 
following argument in \cite{Ton-2} 
works verbatim with 
obvious modifications of changing the domain of 
integration to the one with center at the end-time
point from the center of the space-time domain. 
The second order Taylor expansion of the blow-up
should be changed to the end-time point as well.
The end result is the estimate away from the 
parabolic boundary, as stated in the claim.
\end{proof}

 \appendix

\section{Gaussian density lower bound}
\label{a:densities and tangent flows}
We include the following Lemma for the reader's 
convenience. The localized version can 
be proved similarly. 
\begin{lemma}
Suppose that $\mathscr{V}=\{V_t\}_{t\in(a,b]}$ is a Brakke
flow as in Definition \ref{brakke-def} and that ${\rm spt}\|V_t\|
\subset B_R$ for every $t\in(a,b]$ for some $R>0$.
Then for any $(x_0,t_0)\in{\rm spt}(\|V_t\|
\times dt)$, we have $\Theta(\mathscr{V},(x_0,t_0))\geq 1$.
\end{lemma}
\begin{proof}
The proof is by a contradiction argument. 
If $\Theta(\mathscr{V},(x_0,t_0))<1$, by the definition of the
Gaussian density
and the continuity of the integrand,
there exist some $\tau_0>0$, $\delta_0>0$  and $\varepsilon_0>0$ such that
$|(x_0,t_0)-(x',t')|<\varepsilon_0$ implies
\begin{equation}\label{ap0}
    \frac{1}{(4\pi(t'-t_0+\tau_0))^{\frac{k}{2}}}
    \int_{\mathbb R^n} \exp\left(-\frac{|y-x'|^2}{4(t'-t_0+\tau_0)}\right)\,d\|V_{t_0-\tau_0}\|(y)<1-\delta_0.
\end{equation}
By the definition of Brakke flow, 
we can choose an arbitrarily close point $(x',t')$ 
to $(x_0,t_0)$ such that $V_{t'}\in
{\bf IV}_k(U)$ and $V_{t'}$ has, at $x'$, 
the approximate tangent space with 
integer-multiplicity, say, $j'\in\mathbb N$. Then,
by the property of the approximate tangent space, one can
prove that
\begin{equation}\label{ap1}
    \lim_{\tau\rightarrow 0+}
    \frac{1}{(4\pi\tau)^{\frac{k}{2}}}
    \int_{\mathbb R^n} \exp\left(-\frac{|y-x'|^2}{4\tau}\right)\,d\|V_{t'}\|(y)=j'.
\end{equation}
In particular, \eqref{ap1} implies that 
\begin{equation}\label{ap2}
    1-\frac{\delta_0}{2}\leq\frac{1}{(4\pi\tau)^{\frac{k}{2}}}
    \int_{\mathbb R^n} \exp\left(-\frac{|y-x'|^2}{4\tau}\right)\,d\|V_{t'}\|(y)
\end{equation}
for all sufficiently small $\tau>0$. Since $t'$ may be
arbitrarily close to $t_0$, we may assume that $t'>t_0-\tau_0$,
and by the monotonicity and \eqref{ap2}, we have
\begin{equation}\label{ap3}
    1-\frac{\delta_0}{2}\leq \frac{1}{(4\pi(t'+\tau-t_0+\tau_0))^{\frac{k}{2}}}
    \int_{\mathbb R^n}\exp\left(-\frac{|y-x'|^2}{4(t'+\tau-t_0+\tau_0)}\right)
    \,d\|V_{t_0-\tau_0}\|(y).
\end{equation}
Since $\tau$ is arbitrarily small, we may assume $|(x_0,t_0)-(x',t'+\tau)|<\varepsilon_0$, and \eqref{ap3} is a contradiction to \eqref{ap0}. This proves the claim.
\end{proof}

%\medskip
%
%\noindent \textbf{Conflict of interest:} On behalf of all authors, the corresponding author states that there is no conflict of interest. 
%
%\smallskip
%
%\noindent \textbf{Data availability statement:} Data sharing is not applicable to this article as no datasets were generated or analysed during the current study.

\bibliographystyle{abbrv}
\bibliography{MCF_Plateau_biblio}
\end{document}